\documentclass[11pt]{amsart} 

\usepackage{epigamath}

\usepackage[english]{babel}

\numberwithin{equation}{section}

\usepackage{enumitem}

\usepackage{amsfonts}
\usepackage{amsmath}
\usepackage{amssymb}
\usepackage{amscd}
\usepackage{array}
\usepackage{subfigure}
\usepackage{tikz}
\usetikzlibrary{arrows}
\usepackage[all]{xy}
\usepackage{amsxtra} 

\usepackage{macro-longdash}

\usepackage{url}

\theoremstyle{plain}

\newtheorem{thm}{Theorem}[section]
\newtheorem{pro}[thm]{Proposition}
\newtheorem{lem}[thm]{Lemma}
\newtheorem{cor}[thm]{Corollary}

\theoremstyle{definition}
\newtheorem{dfn}[thm]{Definition}
\newtheorem{nt}[thm]{Notation}

\theoremstyle{remark}
\newtheorem{rem}[thm]{Remark}
\newtheorem{cons}[thm]{Construction}
\newtheorem*{exa*}{Examples}

\allowdisplaybreaks[3]

\newcommand{\loc}{\mathrm{loc}}

\newcommand{\cN}{\mathcal{N}}

\newcommand{\kS}{\mathfrak{S}}

\newcommand{\vir}{\mathrm{vir}}

\newcommand{\cE}{\mathcal{E}}

\newcommand{\aA}{\mathbb{A}}

\newcommand{\PP}{\mathbb{P}}
\newcommand{\CC}{\mathbb{C}}
\newcommand{\ZZ}{\mathbb{Z}}
\newcommand{\NN}{\mathbb{N}}

\newcommand{\QQ}{\mathbb{Q}}

\newcommand{\cO}{\mathcal{O}}
\newcommand{\ev}{\mathrm{ev}}

\renewcommand{\(}{\left(}
\renewcommand{\)}{\right)}

\newcommand{\cC}{\mathcal{C}}

\newcommand{\cF}{\mathcal{F}}

\newcommand{\cM}{\mathcal{M}}

\newcommand{\cP}{\mathcal{P}}

\newcommand{\cX}{\mathcal{X}}
\newcommand{\cY}{\mathcal{Y}}

\newcommand{\Hom}{\operatorname{Hom}}

\def\<{\left\langle}
\def\>{\right\rangle}

\newcommand{\lra}{\longrightarrow}
\newcommand{\longhookrightarrow}{\lhook\joinrel\longrightarrow}

\newcommand{\supth}[1]{\ensuremath{#1^{\mathrm{th}}}}

\EpigaVolumeYear{10}{2026} \EpigaArticleNr{3} \ReceivedOn{July 22, 2022}
\InFinalFormOn{September 17, 2025}
\AcceptedOn{October 3, 2025}

\title{Hodge--Gromov--Witten theory}
\titlemark{HGW theory}

\author{J\'er\'emy Gu\'er\'e}
\address{Institut Fourier, UMR 5582, Laboratoire de Math\'ematiques, Universit\'e Grenoble Alpes, CS 40700, 38058 Grenoble cedex 9, France} 
\email{jeremy.guere@univ-grenoble-alpes.fr}

\authormark{J.~Gu\'er\'e}

\AbstractInEnglish{We determine the all-genus Hodge--Gromov--Witten theory of a smooth hypersurface in weighted projective space defined by a chain or loop polynomial.  In particular, we obtain the first genus zero computation of Gromov--Witten invariants for hypersurfaces in non-Gorenstein ambiant spaces, where the convexity property fails. We extend it to any weighted projective hypersurface defined by an invertible polynomial.}



\begin{document}



\maketitle

\begin{prelims}

\DisplayAbstractInEnglish

\bigskip

\DisplayKeyWords

\medskip

\DisplayMSCclass

\end{prelims}


\newpage

\setcounter{tocdepth}{1}

\tableofcontents


\section{Introduction}
Gromov--Witten theory has seen tremendous development in the last thirty years.  Originating from theoretical physics, it is mathematically formulated as an intersection theory of complex curves traced on a complex smooth projective variety, and provides invariants that one thinks of as a virtual count of these curves.  The most famous example is a full computation of the genus zero invariants enumerating rational curves on the quintic threefold; see \cite{Candelas,Gi,LLY}.

Gromov--Witten theory is well understood in all genera for toric varieties, or even toric Deligne--Mumford (DM) stacks; see \cite{GraberPandha,Liu}. Precisely, the moduli space of stable maps inherits a torus action from the target space, and the computation essentially reduces to a calculation on the moduli space to the fixed locus. This is the content of the virtual localization formula, see \cite{GraberPandha}, which is an enhancement of the classical Atiyah--Bott localization formula; see \cite{Atiyah}. We also refer to \cite{Edidin2} for an algebraic proof.

Smooth hypersurfaces in toric DM stacks are the next class of spaces to consider, but little is known in this situation.  The difficulty comes from the non-invariance of the hypersurface by the torus action in general, so that there is no direct way to apply a localization formula to decrease the complexity of the problem.  Consider the famous example of the quintic hypersurface in $\PP^4$. As we mention above, the genus zero theory is fully determined; see \cite{Candelas,Gi,LLY}.  The genus one case was completely proven by Zinger \cite{Zi} after a great deal of hard work, and nowadays several approaches are solving it up to genus three; see \cite{Honglu5,Ruan5,Li5}.  It is worth noticing that physicists have predictions up to genus $52$, see \cite{Klemm}, and that Maulik--Pandharipande \cite{MaulikPandha} described a proposal working in any genus, although it is too hard to implement for practical use.  We also mention a recent breakthrough proving the holomorphic anomaly equation conjectured by \cite{BCOV}; see \cite{Ruan5,Li5}.

Even in genus zero, the problem of computing Gromov--Witten invariants of smooth hypersurfaces in toric DM stacks is far from being completely solved.  Consider the special case of hypersurfaces in weighted projective spaces. The genus zero theory is only known under a restrictive condition: the degree of the hypersurface is a multiple of every weight. One refers to it as the Gorenstein condition, as it is the condition for the coarse space of the hypersurface to have Gorenstein singularities.  We recall that Gromov--Witten theory is invariant under smooth deformations; hence we can choose any defining polynomial of degree~$d$ as long as the associated hypersurface is a smooth DM stack.  As a consequence, one can also rephrase the Gorenstein condition as the existence of a Fermat hypersurface of degree $d$, which is defined by a Fermat polynomial of the form $x_1^{a_1}+\dotsb+x_N^{a_N}$.

There is a substantial simplification for the genus zero theory of hypersurfaces in weighted projective spaces under the Gorenstein condition. It is called the convexity property; see \cite[Introduction]{Guere1}.  It implies that the virtual cycle of the theory, which is the crucial object to handle, equals the top Chern class of a vector bundle over the moduli space of stable maps to the weighted projective space. It is then calculated by a Grothendieck--Riemann--Roch formula, see \cite{Kim,Lee,Coates,Tseng}, and the genus zero Gromov--Witten theory of the hypersurface is deduced from the genus zero Gromov--Witten theory of the weighted projective space; one calls it the quantum Lefschetz principle; see \cite{Kim,Lee,Coates,Tseng}.  Without the Gorenstein condition, the convexity property may fail and the virtual cycle is not computable.  Although Fan and Lee \cite{Honglu5} obtain a version of the quantum Lefschetz principle in higher genus for projective hypersurfaces, a general statement is false; see \cite{QLPmayfail}.

In this paper, we work on smooth (as DM stacks) hypersurfaces in weighted projective spaces, under a mild condition. Precisely, we relax the existence of a Fermat hypersurface to the existence of a chain hypersurface, which is defined by a chain polynomial of the form $x_1^{a_1}x_2+\dotsb+x_{N-1}^{a_{N-1}}x_N+x_N^{a_N}$, or to the existence of a loop hypersurface, which is defined by a loop polynomial of the form $x_1^{a_1}x_2+\dotsb+x_{N-1}^{a_{N-1}}x_N+x_N^{a_N}x_1$. This is more general than the Gorenstein condition, and non-convex cases appear.  In genus zero, the condition is even more general as we can relax the condition on weights and degree to the existence of a hypersurface defined by an invertible polynomial; see the beginning of Section~\ref{secinv}.  We then prove two results for these hypersurfaces:
\begin{itemize}
\item a genus zero quantum Lefchetz principle; see Corollaries~\ref{0HGWchain} and~\ref{0HGWloop} and Theorem~\ref{HGWinv}; 
\item a Hodge quantum Lefchetz principle in arbitrary genus; see Theorems~\ref{HGWchain} and~\ref{HGWloop}.
\end{itemize}
In the first, we express the genus zero Gromov--Witten theory of the hypersurface in terms of the genus zero Gromov--Witten theory of the weighted projective space.  In the second, we do the same in arbitrary genus, once we cap virtual cycles with the Hodge class, that is, the top Chern class of the Hodge bundle; see Definition~\ref{Hodge vir}.  As a consequence, this paper gives the first computation of the genus zero Gromov--Witten theory of hypersurfaces in a range of cases where the convexity property fails; see for instance Remark~\ref{chi6} for a discussion on Calabi--Yau $3$-folds with Euler characteristic equal to $\pm 6$.  It also gives the first comprehensive computation of Hodge integrals, which are Gromov--Witten invariants involving the Hodge class, in arbitrary genus for chain or loop hypersurfaces.

In order to tackle non-convexity issues, we develop in this paper a method that we phrase in a general framework, opening the way to further new results in Gromov--Witten theory.  We call it \textit{regular specialization} (see Theorem~\ref{skid}) as it consists in deforming a given smooth DM stack into a singular one in a regular way.  It can be understood as an enhancement of the invariance of Gromov--Witten theory under smooth deformations.  Precisely, given a regular family $\cX$ of DM stacks over $\aA^1$, that is, a flat morphism $\cX \to \aA^1$ with $\cX$ smooth, the perfect obstruction theory on the moduli space of stable maps to the total space $\cX$ pulls back to a perfect obstruction theory on every fiber, and the associated virtual cycle is independent of the fiber; we call it a \textit{regularized} virtual cycle. Furthermore, on smooth fibers, it equals the cap product of the Gromov--Witten cycle with the Hodge class.  Provided we have a global torus action on the family $\cX$ over $\aA^1$, the regularized virtual cycle localizes to the fixed locus in the central fiber; see Theorem~\ref{skid equiv}.

Under special assumptions listed at the beginning of Section~\ref{QLsection}, we prove a version of an equivariant quantum Lefschetz theorem, see Theorem~\ref{quantum Lefschetz}, relating for a $\CC^*$-equivariant embedding $\cX \hookrightarrow \cP$ of DM stacks the equivariant virtual cycles associated to $\cX$ and to $\cP$.  Moreover, if a torus $T=\(\CC^*\)^r$ acts on the ambient space $\cP$, \textit{e.g.}~if it is a toric DM stack, then we can relate the $\CC^*$-virtual cycle associated to $\cX$ to the $T$-virtual cycle associated to $\cP$.  Together with the regular specialization theorem, it yields Theorem~\ref{big theo}.

Genus zero is a special interesting case, as the Hodge class equals the fundamental class and the regularized virtual cycle equals the Gromov--Witten virtual cycle.  Let us call a DM stack regularizable, see Definition~\ref{regularisable}, if we can embed it as a fiber of a family of DM stacks over the affine space, whose total space is regular.  Although a regularizable DM stack may have bad singularities, we provide it with a genus zero Gromov--Witten theory via the regularized virtual cycle, and we prove the invariance of the genus zero theory under regular deformations; see Proposition~\ref{invariance}.  As a consequence, we can apply the localization formula whenever we have a torus action on the fiber, not necessarily on the total family.  One strategy to compute the genus zero Gromov--Witten theory of a DM stack is thus to take a regular specialization to another DM stack admitting a torus action with sufficiently nice fixed locus; see below for more details.

Finally, we highlight that this paper provides answers in Gromov--Witten theory to similar questions in FJRW theory (see \cite{FJRW1, FJRW2}) answered in \cite{Guere1,Guere2,Guere9}.  More precisely, it enters the big picture of the Landau--Ginzburg/Calabi--Yau (LG/CY) correspondence; see \cite{LGCY}.  In particular, Theorem~\ref{0HGWchain} should lead to a computation of the I-function using Givental's formalism, see \cite{Givental2},  and eventually to a genus zero mirror symmetry theorem without convexity.  Comparing with results in \cite{Guere1}, we should then obtain the LG/CY correspondence, extending the work of Chiodo--Iritani--Ruan \cite{LGCY}. We will discuss it in another paper.  We also observe that the knowledge of Hodge integrals is crucial for a computation of the Hamiltonians of the double ramification (DR) hierarchy introduced by Buryak \cite{Buryak} and may lead to new insights on the structure of Gromov--Witten invariants.

\subsubsection*{A note on tautological classes}  A fundamental question about virtual cycles is whether their push-forward to the moduli space of stable curves lies in the tautological Chow ring.  This question is largely open; \textit{e.g.}~it is unknown in the case of the quintic threefold.  A straightforward and yet noticeable consequence of our results is that the product of the Hodge class $\lambda_g$ with the virtual cycle is tautological in the Chow ring of $\overline{\cM}_{g,n}$ in all the cases we study, \textit{e.g.}~for smooth hypersurfaces defined by chain or loop polynomials.  This follows from the virtual localization formula and from the fact that fixed loci in the target space are isolated points.

\subsubsection*{Future works}  Based on this paper, we have thought of a new strategy aiming at computing all-genus Gromov--Witten invariants of projective hypersurfaces, and possibly other projective varieties.  The idea is the following: by Costello's theorem, see \cite{Costello}, genus $g$ Gromov--Witten invariants of a smooth projective variety $X$ are explicitly expressed in terms of genus zero Gromov--Witten invariants of the symmetric product $S^{g+1}X$.

Let $X$ be a smooth complex projective variety, and assume we have a family $\cX$ of DM stacks over $\aA^1$ admitting a torus action and whose fiber at $1 \in \aA^1$ is $X$.  Taking the symmetric fibered product over $\aA^1$, we obtain a family $\cX_g$ of DM stacks over $\aA^1$ admitting a torus action and whose fiber at $1 \in \aA^1$ is the smooth DM stack $S^{g+1}X$.  Precisely, we have
$$
\cX_g = \left[\cX \times_{\aA^1} \dotsm \times_{\aA^1} \cX / \kS_{g+1}\right].
$$
By Hironaka's theorem, see \cite{Hironaka} (note that we work over $\CC$), and its equivariant version (see \textit{e.g.}~\cite{Kollar}), there exists a resolution of singularities $\widetilde{\cX_g}$ of the DM stack $\cX_g$, which is an isomorphism outside the singular locus of the DM stack $\cX_g$ and which preserves the torus action. In particular, we get a morphism $\widetilde{\cX_g} \to \aA^1$, and the fiber at $1 \in \aA^1$ is still $S^{g+1}X$, because it is in the smooth locus of $\cX_g$.  Moreover, the birational map $\widetilde{\cX_g} \to \cX_g$ is obtained by a sequence of blow-ups, and the morphism $\cX_g \to \aA^1$ is flat; hence the morphism $\widetilde{\cX_g} \to \aA^1$ is flat as well; see for instance \cite[Appendix B.6.7]{Fulton}.  As a consequence, the DM stack $\widetilde{\cX_g}$ is a regular family over $\aA^1$ admitting a torus action and whose fiber at $1 \in \aA^1$ is the symmetric product $S^{g+1}X$.  According to our genus zero regular specialization theorem, genus zero Gromov--Witten invariants of $S^{g+1}X$, and thus genus $g$ Gromov--Witten invariants of $X$, are expressed by the localization formula in terms of genus zero Gromov--Witten invariants of the torus-fixed loci in (the fiber at $0 \in \aA^1$ of) $\widetilde{\cX_g}$.

\subsubsection*{Related works}  Since the first version of this paper, several other results have appeared regarding non-convex Gromov--Witten theory computations. We are aware of the following works: \cite{JunWang}, \cite{Shoemaker}, and \cite{Janda}.  All these works have different approaches from the one presented in this paper, and it would be interesting to compare them and show that they indeed compute the same numbers. Better, it would be great if we can combine these various methods to go even further in understanding non-convexity.

\subsection*{Acknowledgments}
The author is grateful to Alessandro Chiodo, Rahul Pandharipande, and Honglu Fan for many interesting discussions on this topic.  He would also like to thank his wife and daughters for their help and understanding during the paper's finalization.

\section{Hodge--Gromov--Witten theory}\label{sec1}
In this section, we prove a general theorem on Hodge--Gromov--Witten theory, which we call the `regular specialization theorem'. The context is the following.

\begin{dfn}\label{Hodge vir}
Given a smooth DM stack $\cY$, Gromov--Witten theory provides a virtual fundamental cycle for the moduli space $\cM_{\cY}$ of stable maps to $\cY$.  We will call the cup product of the virtual fundamental cycle with the top Chern class of the Hodge bundle\footnote{For a family $\pi \colon \cC \to S$ of genus $g$ curves, the Hodge bundle is a rank $g$ vector bundle on $S$ defined by the push-forward $\pi_* \omega_{\cC/S}$ of the relative canonical sheaf.} the \textit{Hodge virtual cycle}.  Hodge--Gromov--Witten theory is then intersection theory on $\cM_{\cY}$ against this cycle.
\end{dfn}

\begin{dfn}
A morphism $f \colon \cX \to \cY$ between two DM stacks is called a family when it is flat.  We also say that $\cX$ is a $\cY$-family.  Inverse images of geometric points $y \in \cY$ are called fibers.  A regular family is a family for which the DM stack $\cX$ is smooth.
\end{dfn}

Let $p \colon \cX \to \aA^1$ be a regular family of DM stacks over $\aA^1$, and denote by $X_0$ and $X_1$ its fibers at $0 \in \aA^1$ and at $1 \in \aA^1$.  We assume $X_0$ and $X_1$ to be proper, and $X_1$ to be smooth, but we do not impose any restriction on singularities of $X_0$.  Depending on the purpose, we may also assume the family $\cX$ is equipped with a torus action leaving $X_0$ invariant.

Let $\cM_{X_0}$ and $\cM_{X_1}$ be the moduli spaces of stable maps to $X_0$ and to $X_1$, with arbitrary genus, degree, number of markings, and isotropy type at markings.  Gromov--Witten theory for smooth DM stacks provides a perfect obstruction theory and a virtual fundamental cycle for the moduli space $\cM_{X_1}$, but not for $\cM_{X_0}$.

In the first subsection, we define the `regularized virtual cycles' for the moduli space $\cM_{X_1}$.  Moreover, we show that it equals the Hodge--Gromov--Witten virtual cycle, up to a sign.

Then, we define the regularized virtual cycle for the moduli space $\cM_{X_0}$.  The regular specialization theorem can be phrased as an equality between regularized virtual cycles of $\cM_{X_1}$ and of $\cM_{X_0}$.

Graber--Pandharipande's virtual localization formula, see \cite{GraberPandha}, applies to regularized virtual cycles.  Therefore, provided we have a torus action on the family preserving the central fiber $X_0$ and since the fixed locus in $X_0$ is smooth, we can decompose the Hodge--Gromov--Witten cycle of $\cM_{X_1}$ into Hodge--Gromov--Witten cycles of the fixed loci in $\cM_{X_0}$.

\subsubsection*{Conventions}
In this whole section, we fix the topological data of stable maps, such as the genus $g$, the number of markings $n$, the curve class $\beta$, and the isotropies, even though these topological data may be absent from the notation.  Furthermore, all perfect obstruction theories will be relative to the moduli stack of prestable curves $\mathfrak{M}_{g,n}$. We abusively omit it in the notation to make it more readable.

\subsection{(Relative) perfect obstruction theory for $\boldsymbol{X_1}$}\label{secpot}

\begin{nt}
For a DM stack $\cY$, we denote by $\cM_{\cY}$ the moduli space of stable maps to $\cY$, by $\pi_\cY \colon \cC_\cY \to \cM_\cY$ the universal curve, by $f_\cY \colon \cC_\cY \to \cY$ the universal map, and by $\omega_{\pi_{\cY}}$ the relative dualizing sheaf.  In the special cases of $\cX$, $X_0$, and $X_1$, we simplify the notation of the maps as
$$\pi=\pi_{\cX},\quad \pi_0=\pi_{X_0},\quad \pi_1=\pi_{X_1},\quad f=f_{\cX},\quad f_0=f_{X_0},\quad f_1=f_{X_1}.$$
\end{nt}

The morphism $p \colon \cX \to \aA^1$ induces a morphism
$$q \colon \cM_{\cX} \lra \cM_{\aA^1} \simeq \aA^1 \times \overline{\cM}_{g,n}.$$
Furthermore, we have fiber diagrams
\begin{equation}\label{diagrams}
\begin{tikzpicture}[scale=0.75]	
\node (A) at (0,1.5) {$\cM_{X_0}$};
\node (B) at (2.5,1.5) {$\cM_{\cX}$};
\node (C) at (2.5,0) {$\cM_{\aA^1}$,};
\node (D) at (0,0) {$\overline{\cM}_{g,n}$};
\draw[->] (D) -- (C);
\draw[->] (A) -- (B);
\draw[->] (A) -- (D);
\draw[->] (B) -- (C);
\draw[-] (1.25-0.1,0.75-0.1) -- (1.25-0.1,0.75+0.1) -- (1.25+0.1,0.75+0.1) -- (1.25+0.1,0.75-0.1) -- (1.25-0.1,0.75-0.1);
\node[above] at (1.25,1.5){$j_0$};
\node[right] at (2.5,0.75){$q$};
\node[left] at (0,0.75){$q_0$};
\node (A') at (6+0,1.5) {$\cM_{X_1}$};
\node (B') at (6+2.5,1.5) {$\cM_{\cX}$};
\node (C') at (6+2.5,0) {$\cM_{\aA^1}$,};
\node (D') at (6+0,0) {$\overline{\cM}_{g,n}$};
\draw[->] (D') -- (C');
\draw[->] (A') -- (B');
\draw[->] (A') -- (D');
\draw[->] (B') -- (C');
\draw[-] (6+1.25-0.1,0.75-0.1) -- (6+1.25-0.1,0.75+0.1) -- (6+1.25+0.1,0.75+0.1) -- (6+1.25+0.1,0.75-0.1) -- (6+1.25-0.1,0.75-0.1);
\node[above] at (6+1.25,1.5){$j_1$};
\node[right] at (6+2.5,0.75){$q$};
\node[left] at (6+0,0.75){$q_1$};
\end{tikzpicture}
\end{equation}
where the  bottom arrows are inclusions of $0 \in \aA^1$ and $1 \in \aA^1$.
In particular, the maps $j_0$ and $j_1$ are closed immersions, hence proper.
We also introduce notation for maps in the following fiber diagrams: 
\begin{center}
\begin{tikzpicture}[scale=0.75]	
\node (A) at (0,1.5) {$X_0$};
\node (B) at (2.5,1.5) {$\cX$};
\node (C) at (2.5,0) {$\aA^1$,};
\node (D) at (0,0) {$0$};
\draw[->] (D) -- (C);
\draw[->] (A) -- (B);
\draw[->] (A) -- (D);
\draw[->] (B) -- (C);
\draw[-] (1.25-0.1,0.75-0.1) -- (1.25-0.1,0.75+0.1) -- (1.25+0.1,0.75+0.1) -- (1.25+0.1,0.75-0.1) -- (1.25-0.1,0.75-0.1);
\node[above] at (1.25,1.5){$i_0$};
\node[right] at (2.5,0.75){$p$};
\node[left] at (0,0.75){$p_0$};
\node (A') at (6+0,1.5) {$X_1$};
\node (B') at (6+2.5,1.5) {$\cX$};
\node (C') at (6+2.5,0) {$\aA^1$.};
\node (D') at (6+0,0) {$1$};
\draw[->] (D') -- (C');
\draw[->] (A') -- (B');
\draw[->] (A') -- (D');
\draw[->] (B') -- (C');
\draw[-] (6+1.25-0.1,0.75-0.1) -- (6+1.25-0.1,0.75+0.1) -- (6+1.25+0.1,0.75+0.1) -- (6+1.25+0.1,0.75-0.1) -- (6+1.25-0.1,0.75-0.1);
\node[above] at (6+1.25,1.5){$i_1$};
\node[right] at (6+2.5,0.75){$p$};
\node[left] at (6+0,0.75){$p_1$};
\end{tikzpicture}
\end{center}
The map $i_1$ yields an exact triangle of cotangent complexes
$$i_1^*L_{\cX} \lra L_{X_1} \lra L_{X_1/\cX} \lra i_1^*L_{\cX}[1].$$
From the construction of obstruction theories on moduli spaces of maps, we obtain a commutative diagram
\begin{equation}
\begin{tikzpicture}[scale=1]	
\node (A) at (0,1) {$j_1^* E_{\cX}$};
\node (B) at (2.5,1) {$E_{X_1}$};
\node (C) at (5,1) {$E_{X_1/\cX}$};
\node (D) at (7.5,1) {$j_1^*E_{\cX}[1]$};
\draw[->] (A) -- (B);
\draw[->] (B) -- (C);
\draw[->] (C) -- (D);
\node (A') at (0,0) {$j_1^*L_{\cM_{\cX}}$};
\node (B') at (2.5,0) {$L_{\cM_{X_1}}$};
\node (C') at (5,0) {$L_{\cM_{X_1}/\cM_{\cX}}$};
\node (D') at (7.5,0) {$j_1^*L_{\cM_{\cX}}[1]$,};
\draw[->] (A') -- (B');
\draw[->] (B') -- (C');
\draw[->] (C') -- (D');
\draw[->] (A) -- (A');
\draw[->] (B) -- (B');
\draw[->] (C) -- (C');
\draw[->] (D) -- (D');
\end{tikzpicture}
\end{equation}
where each row is an exact triangle and where the obstruction theories are defined as
\begin{align*}
	E_{\cX} & :=  R\pi_*\(f^*L_{\cX} \otimes \omega_{\pi_{\cX}}\) \simeq \(R\pi_*f^*T_{\cX}\)^\vee, \\
	E_{X_1} & :=  R{\pi_1}_*\(f_1^*L_{X_1} \otimes \omega_{\pi_{X_1}}\) \simeq \(R{\pi_1}_*f_1^*T_{X_1}\)^\vee, \\
	E_{X_1/\cX} & :=  R{\pi_1}_*\(f_1^*L_{X_1/\cX} \otimes \omega_{\pi_{X_1}}\) \simeq R{\pi_1}_*\(\omega_{\pi_{X_1}}\)[1],
\end{align*}
where the smoothness of $\cX$ and of $X_1$ is used in the second equality of the first two lines.  For the second equality of the third line, we use that $p$ is flat to compute
$$L_{X_0/\cX} \simeq p_0^*L_{0/\aA^1} = \cO[1].$$

\begin{rem}\label{perfect}
  Since the stacks $\cX$ and $X_1$ are smooth, the obstruction theories $E_{\cX}$ and $E_{X_1}$ are perfect, \textit{i.e.}~of
  amplitude in $[-1,0]$.  On the other hand, the obstruction theory $E_{X_1/\cX}$ is of amplitude in $[-2,-1]$; hence it is not perfect.  However, one can represent it by a two-term complex of vector bundles $[A \to B]$ whose kernel and cokernel are given by
  $$0 \lra \mathbb{E} \lra A \lra B \lra \cO \lra 0.$$
Here, we recall that $\mathbb{E}:={\pi_1}_*(\omega_{\pi_{X_1}})$ is the Hodge bundle and is a pull-back from the moduli space of stable curves $\overline{\cM}_g$.  In particular, we have a map of complexes
$$\mathbb{E}[2] \lra E_{X_1/\cX}$$
that we use in the following definition.
\end{rem}

\begin{dfn}\label{skid potX_1}
The (relative) regularized obstruction theory for $\cM_{X_1}$ is defined as follows.  We first take
$$F_{X_1} := E_{X_1} \oplus \mathbb{E}[1],$$
and then use the map $E_{X_1} \to L_{\cM_{X_1}}$ and the composed morphism
$$\mathbb{E} \lra E_{X_1/\cX}^{-2} \lra \(j_1^*E_{\cX}\)^{-1} \lra \(j_1^*L_{\cM_{\cX}}\)^{-1} \lra L^{-1}_{\cM_{X_1}}$$
to get $F_{X_1} \to L_{\cM_{X_1}}.$ Clearly, it is a perfect obstruction theory on $\cM_{X_1}$.
\end{dfn}

\begin{dfn}\label{skid virtual cycle}
We call the virtual fundamental cycle $[\cM_{X_1},F_{X_1}] \in A_*(\cM_{X_1})$ obtained by Behrend--Fantechi \cite{BF} from the relative perfect obstruction theory $F_{X_1}$ (and the perfect obstruction theory from prestable curves) the \textit{regularized virtual cycle} of $\cM_{X_1}$.

We also call the virtual fundamental cycle $[\cM_{X_1},E_{X_1}] \in A_*(\cM_{X_1})$ obtained by Behrend--Fantechi from the perfect obstruction theory $E_{X_1}$ (and the perfect obstruction theory from prestable curves) the \textit{Gromov--Witten virtual cycle} of $\cM_{X_1}$.
\end{dfn}

\begin{lem}\label{Hodge}
The regularized virtual cycle equals the Hodge--Gromov--Witten virtual cycle up to a sign.  Precisely, we have the relation
$$
\left[\cM_{X_1},F_{X_1}\right] = (-1)^g \lambda_g \cdot \left[\cM_{X_1},E_{X_1}\right] \in A_*\(\cM_{X_1}\),
$$
where $\lambda_g := c_\mathrm{top}(\mathbb{E})$ is the top Chern class of the Hodge bundle.
\end{lem}

\begin{proof}
The virtual fundamental class $[\cM_{X_1},E_{X_1}]$ is the intersection of the intrinsic normal cone $\mathfrak{C}_{\cM_{X_1}}$ of $\cM_{X_1}$ with the zero section of $h^1/h^0(E_{X_1}^\vee)$, and similarly for $[\cM_{X_1},F_{X_1}]$.  Since $F_{X_1} := E_{X_1} \oplus \mathbb{E}[1]$, we get
$$h^1/h^0\(F_{X_1}^\vee\) \simeq h^1/h^0\(E_{X_1}^\vee\) \times \mathrm{Spec}(\mathrm{Sym}\mathbb{E}).$$
Therefore, we have
\begin{align*}\pushQED{\qed}
\left[\cM_{X_1},F_{X_1}\right] &\enspace = \enspace 0^!_{h^1/h^0(F_{X_1}^\vee)}\left[\mathfrak{C}_{\cM_{X_1}}\right] \\
&\enspace = \enspace 0^!_{\mathrm{Spec}(\mathrm{Sym}\mathbb{E})}0^!_{h^1/h^0(E_{X_1}^\vee)}\left[\mathfrak{C}_{\cM_{X_1}}\right] \\
&\enspace = \enspace 0^!_{\mathrm{Spec}(\mathrm{Sym}\mathbb{E})} \left[\cM_{X_1},E_{X_1}\right] \\
&\enspace = \enspace c_\mathrm{top}\(\mathbb{E}^\vee\) \cap\left[\cM_{X_1},E_{X_1}\right].
\hfill\qedhere \popQED
\end{align*}
\renewcommand{\qed}{}    
\end{proof}

\begin{rem}\label{bad news}
  The reader should be careful that the obstruction theory $E_{X_1/\cX}$ is not equal in the derived category to its cohomology $\mathbb{E}[2] \oplus \cO[1]$, as we could have first believed.  However, we have the equality once we pass to K-theory, and thus to Chow rings/cohomology.  The same happens for $E_{X_0/\cX}$.  As a consequence, we are unable to define a perfect obstruction theory for $X_0$ (at the derived level).  However, a regularized virtual cycle (or even a regularized K-theoretic structure sheaf) does exist for $X_0$. We define it in the next subsection.
\end{rem}

\subsection{Regular specialization theorem}
First, we define the regularized virtual cycle of $\cM_{X_0}$, and we compare it to that of $\cM_{X_1}$. Then, we discuss the corresponding invariants.

\subsubsection{Pull-backs from the regular family}

\begin{dfn}\label{0skid virtual cycle}
We define the \textit{regularized virtual cycle} of $\cM_{X_0}$ as the pull-back of the virtual cycle from the regular total family $\cX$, \textit{i.e.}
$$\left[\cM_{X_0},F_{X_0}\right] := j_0^! \left[\cM_{\cX},E_{\cX}\right] \in A_*\(\cM_{X_0}\).$$
Note that the left-hand side is just notation. In particular, it is \textit{a priori} not induced by a perfect obstruction theory $F_{X_0}$ on $\cM_{X_0}$; see Remark~\ref{bad news}.
\end{dfn}

The following proposition shows that it is indeed the right thing to do.

\begin{pro}\label{equality of skid}
The regularized virtual cycle associated to a fiber of a regular family over $\aA^1$ does not depend on the fiber.  Precisely, we have equalities
\begin{align*}
j_0^! \left[\cM_{\cX},E_{\cX}\right] & =  \left[\cM_{X_0},F_{X_0}\right] \in A_*\(\cM_{X_0}\), \\
j_1^! \left[\cM_{\cX},E_{\cX}\right] & =  \left[\cM_{X_1},F_{X_1}\right] \in A_*\(\cM_{X_1}\).
\end{align*}
\end{pro}

\begin{proof}
The first equality is just the definition itself.  For the second equality, we observe the following.

Restricting to some open neighbourhood $U$ of $1$ in $\mathbb{A}^1$, we can assume that the map $p_{|U} \colon \cX_{|U} \to U$ is smooth.  Therefore, the (usual) perfect obstruction theory of $\cM_{X_1}$ coincides with the relative perfect obstruction theory of $q_{|U} \colon \cM_{\cX_{|U}} \to \cM_U$.  The (absolute, yet relative to the moduli of prestable curves as always) perfect obstruction theory of $\cM_{\cX_{|U}}$ is then obtained by composing it with the perfect obstruction theory of $\cM_U$, whose $h^0$ is $\cO$ and $h^{-1}$ is the Hodge bundle $\mathbb{E}$.

We further decompose the perfect obstruction theory of $\cM_U$ along the projection
$$\cM_U \simeq \overline{\cM}_{g,n} \times U \lra U.$$
This relative perfect obstruction theory is given by the obstruction vector bundle $\mathbb{E}^\vee$, because the moduli space is smooth.  As a consequence, the perfect obstruction theory of $\cM_{\cX_{|U}}$ relative to $U$ coincides exactly with $F_{X_1}$, so that we get the equality
\begin{equation*}\pushQED{\qed}
j_1^! \left[\cM_{\cX},E_{\cX}\right] = \left[\cM_{X_1},F_{X_1}\right].\qedhere \popQED
	\end{equation*}
\renewcommand{\qed}{}  
\end{proof}

\begin{lem}
The morphism $q \colon \cM_{\cX} \to \aA^1$ is proper.  Moreover, we have commutative diagrams
\begin{center}
	\begin{tikzpicture}[scale=0.75]	
	\node (A) at (0,1.5) {$\cM_{X_1}$};
	\node (B) at (2.5,1.5) {$\cM_{\cX}$};
	\node (C) at (2.5,0) {$\cX$,};
	\node (D) at (0,0) {$X_1$};
	\draw[->] (D) -- (C);
	\draw[->] (A) -- (B);
	\draw[->] (A) -- (D);
	\draw[->] (B) -- (C);
	\node[above] at (1.25,1.5){$j_1$};
	\node[right] at (2.5,0.75){$\mathrm{ev}_\cX$};
	\node[left] at (0,0.75){$\mathrm{ev}_{X_1}$};
	\node[below] at (1.25,0){$i_1$};
	\node[] at (1.25,0.75){$\circlearrowleft$};
	\node (A') at (6+0,1.5) {$\cM_{X_1}$};
	\node (B') at (6+2.5,1.5) {$\cM_{\cX}$};
	\node (C') at (6+1.25,0) {$\bigcup_{g,n} \overline{\cM}_{g,n}$\,,};
	\draw[->] (A') -- (B');
	\draw[->] (A') -- (C');
	\draw[->] (B') -- (C');
	\node[above] at (6+1.25,1.5){$j_1$};
	\node[right] at (6+2,0.75){$r_\cX$};
	\node[left] at (6+0.5,0.75){$r_{X_1}$};
	\node[below] at (1.25,0){$i_1$};
	\node[] at (6+1.25,0.85){$\circlearrowleft$};
	\end{tikzpicture}
\end{center}
where the maps $\mathrm{ev}_\cX$ and $\mathrm{ev}_{X_1}$ are the evaluation maps and the maps $r_\cX$ and $r_{X_1}$ remember only the coarse curve and stabilize it.  We also have the same commutative diagrams when we replace $X_1$ by $X_0$.
\end{lem}

\begin{proof}
Since every morphism from a nodal curve to the affine line is a contraction to a point, we then have an isomorphism between the moduli space $\cM_{\cX}$ of stable maps to $\cX$ and the moduli space $\cM_{g,n}(\cX/\aA^1)$ of relative stable maps to the family $p \colon \cX \to \aA^1$.  Therefore, by \cite[Section 8.3]{AbramovichVistoli}, the morphism $q \colon \cM_{\cX} \to \aA^1$ is proper.  The commutativity of the diagrams is obvious.
\end{proof}

\begin{rem}
It is not straightforward to compare curve classes for $\cX$, $X_0$, and $X_1$, because we can have vanishing cycles.  To solve this issue, we introduce an ambient space which contains every fiber of $\cX$.
\end{rem}

\subsubsection{Ambient space}
From now on, we assume we have a smooth proper DM stack $\cP$ with an embedding of $\aA^1$-families $\cX \hookrightarrow \cP \times \aA^1$; \textit{i.e.}~every fiber of $\cX$ lies in $\cP$.  In particular, we have push-forward maps
$$H_2(X_t) \lra H_2(\cP)\quad \text{for every $t \in \aA^1$}.$$
We also have maps $\cM_{X_t} \to \cM_{\cP}$ that we can decompose in terms of curve classes. Precisely, for every $\beta \in H_2(\cP)$, we have
$$\bigsqcup_{\substack{\beta' \in H_2(X_t) ~\textrm{with}\\\beta' = \beta \in H_2(\cP)}} \cM_{X_t}(\beta') \lra \cM_{\cP}(\beta).$$
As a consequence, Proposition~\ref{equality of skid} becomes the following.

\begin{pro}\label{skid 2}
For every genus $g$, number of markings $n$, curve class $\beta \in H_2(\cP)$, and isotropies $\underline{\rho} = (\rho_1, \dotsc, \rho_n)$ in $\cX$, we have
\begin{align*}
j_0^! [\cM_{\cX},E_{\cX}]_{g,n,\beta,\underline{\rho}} & =  \sum_{\substack{\beta_0 \in H_2(X_0) ~\textrm{with} \\ \beta_0 = \beta \in H_2(\cP)}} \left[\cM_{X_0},F_{X_0}\right]_{g,n,\beta_0,\underline{\rho}} \in A_*\(\cM_{X_0}\), \\
j_1^! [\cM_{\cX},E_{\cX}]_{g,n,\beta,\underline{\rho}} & =  \sum_{\substack{\beta_1 \in H_2(X_1) ~\textrm{with} \\ \beta_1 = \beta \in H_2(\cP)}} \left[\cM_{X_1},F_{X_1}\right]_{g,n,\beta_1,\underline{\rho}} \in A_*\(\cM_{X_1}\).
\end{align*}
\end{pro}

\subsubsection{Gromov--Witten cycles}
In this subsection, we fix a genus $g$, a number of markings $n$, isotropies $\underline{\rho} = (\rho_1, \dotsc, \rho_n)$ in $\cX$, and curve classes $\beta \in H_2(\cP)$, $\beta_0 \in H_2(X_0)$, and $\beta_1 \in H_2(X_1)$ satisfying
$${i_0}_*(\beta_0) = \beta \in H_2(\cP) \quad \textrm{and} \quad {i_1}_*(\beta_1) = \beta \in H_2(\cP).$$
We also fix $\alpha_1, \dotsc, \alpha_n \in A^*(I\cX)$ such that
$$\alpha_i \in A^*\(\cX_{\rho_i}\) \subset A^*(I\cX),$$
where $\cX_{\rho_i}$ is the component of the inertia stack of $\cX$ with isotropy $\rho_i$.
Furthermore, we denote by $\psi_i$ the usual psi-class on the moduli space of stable curves, \textit{i.e.}~the first Chern class of the cotangent line of the curve at the $\supth{i}$ marking, and by $\lambda_g$ the top Chern class of the (pull-back of the) Hodge bundle.

\begin{dfn}
A Gromov--Witten cycle of $X_1$ is
$$\left[ \alpha_1,\dotsc,\alpha_n \right]^{X_1}_{g,n,\beta_1} := {q_1}_* \( \left[\cM_{X_1},E_{X_1}\right]_{g,n,\beta_1,\underline{\rho}} \cdot \prod_{i=1}^n j_1^* \mathrm{ev}_{\cX}^*(\alpha_i) \) \in A_*\(\overline{\cM}_{g,n}\).$$
A Hodge--Gromov--Witten cycle of $X_1$ is
$$\left[ \alpha_1,\dotsc,\alpha_n \right]^{X_1,\lambda_g}_{g,n,\beta_1} := {q_1}_* \(\lambda_g \cdot \left[\cM_{X_1},E_{X_1}\right]_{g,n,\beta_1,\underline{\rho}} \cdot \prod_{i=1}^n j_1^*\mathrm{ev}_{\cX}^*(\alpha_i) \) \in A_*\(\overline{\cM}_{g,n}\).$$
A relative Gromov--Witten cycle of $p \colon \cX \to \aA^1$ is
$$\left[ \alpha_1,\dotsc,\alpha_n \right]^{\cX, \mathrm{rel}}_{g,n,\beta} := q_* \( [\cM_{\cX},E_{\cX}]_{g,n,\beta,\underline{\rho}} \cdot \prod_{i=1}^n \mathrm{ev}_{\cX}^*(\alpha_i)\) \in A_*\(\cM_{\aA^1}\) \simeq A_*\(\overline{\cM}_{g,n}\).$$
A regularized Gromov--Witten cycle of $X_0$ is
$$\left[ \alpha_1,\dotsc,\alpha_n \right]^{X_0,\mathrm{reg}}_{g,n,\beta_0} := {q_0}_* \(\left[\cM_{X_0},F_{X_0}\right]_{g,n,\beta_0,\underline{\rho}} \cdot \prod_{i=1}^n j_0^*\mathrm{ev}_{\cX}^*(\alpha_i)\) \in A_*\(\overline{\cM}_{g,n}\).$$
A regularized Gromov--Witten cycle of $X_1$ is
$$\left[ \alpha_1,\dotsc,\alpha_n \right]^{X_1,\mathrm{reg}}_{g,n,\beta_1} := {q_1}_* \( \left[\cM_{X_1},F_{X_1}\right]_{g,n,\beta_1,\underline{\rho}} \cdot \prod_{i=1}^n j_1^*\mathrm{ev}_{\cX}^*(\alpha_i)\) \in A_*\(\overline{\cM}_{g,n}\).$$
\end{dfn}

\begin{dfn}\label{ambient theory}
By \textit{ambient theory}, we mean the special case where we take isotropies $\underline{\rho}$ in $\cP$ and insertions
$$\alpha_i \in A^*\(\cP_{\rho_i}\) \subset A^*(I\cP) \lra A^*(I\cX),$$
where pull-back is taken under the map $\cX \hookrightarrow \cP \times \aA^1 \to \cP$.
\end{dfn}

From Lemma~\ref{Hodge}, we see that
\begin{equation*}
\left[ \alpha_1,\dotsc,\alpha_n \right]^{X_1,\mathrm{reg}}_{g,n,\beta_1} =  (-1)^g \left[ \alpha_1,\dotsc,\alpha_n \right]^{X_1,\lambda_g}_{g,n,\beta_1}, \\
\end{equation*}
and from Proposition~\ref{skid 2}, we obtain
\begin{align}\label{equal}
\left[ \alpha_1,\dotsc,\alpha_n \right]^{\cX,\mathrm{rel}}_{g,n,\beta} \ & =  \sum_{\substack{\beta_1 \in H_2(X_1) ~\textrm{with} \\ \beta_1 = \beta \in H_2(\cP)}} \left[ \alpha_1,\dotsc,\alpha_n \right]^{X_1,\mathrm{reg}}_{g,n,\beta_1} \nonumber \\
& =  \sum_{\substack{\beta_0 \in H_2(X_0) ~\textrm{with} \\ \beta_0 = \beta \in H_2(\cP)}} \left[ \alpha_1,\dotsc,\alpha_n \right]^{X_0,\mathrm{reg}}_{g,n,\beta_0}.
\end{align}
Indeed, the two diagrams in~\eqref{diagrams} are Cartesian, the bottom horizontal map $\overline{\cM}_{g,n} \to \cM_{\aA^1} \simeq \aA^1 \times \overline{\cM}_{g,n}$ is a regular embedding, and the vertical maps are proper.  Therefore, for any cycle $\alpha$ on $\cM_{\cX}$, we get
$${q_k}_*j_k^!(\alpha) = k^*q_*(\alpha), \quad k \in \left\lbrace 0,1\right\rbrace;$$
see \cite[Theorem $6.2$]{Fulton}.
Since the map $k^*$ is the identity in Chow rings,
we get \eqref{equal}.

We sum this up with the following statement.

\begin{thm}[Regular specialization theorem]\label{skid}
Let $\cX$ be a regular family over $\aA^1$ whose fibers are embedded in a smooth proper DM stack $\cP$.  For every genus $g$, number of markings $n$ such that $2g-2+n>0$, isotropies $\underline{\rho} = (\rho_1, \dotsc, \rho_n)$ in $X_1$, curve class $\beta \in H_2(\cP)$, and insertions $\alpha_1, \dotsc, \alpha_n \in A^*(I\cX)$ with $\alpha_i \in A^*({\cX}_{\rho_i})$, in $A_*\(\overline{\cM}_{g,n}\)$ we have 
\begin{equation*}
\sum_{\substack{\beta_1 \in H_2(X_1) ~\textrm{with}\\\beta_1 = \beta \in H_2(\cP)}} \left[ \alpha_1,\dotsc,\alpha_n \right]^{X_1,\lambda_g}_{g,n,\beta_1} = (-1)^g \left[ \alpha_1,\dotsc,\alpha_n \right]^{\cX,\mathrm{rel}}_{g,n,\beta}.
\end{equation*}
\end{thm}

\begin{rem}
We might not have access to all insertions for $X_1$, as it is possible that the pull-back map $A^*(\cX) \to A^*(X_1)$ is not surjective.  However, it is easy to work with the ambient theory, \textit{i.e.}~insertions pulled back from $A^*(I\cP)$.
\end{rem}

\begin{cor}[Regular specialization theorem in genus zero]\label{skidg=0}
Under the same assumptions as before,  in $A_*\(\overline{\cM}_{0,n}\)$ we have 
\begin{equation*}
\sum_{\substack{\beta_1 \in H_2(X_1) ~\textrm{with}\\\beta_1 = \beta \in H_2(\cP)}} \left[ \alpha_1,\dotsc,\alpha_n \right]^{X_1}_{0,n,\beta_1} =\left[ \alpha_1,\dotsc,\alpha_n \right]^{\cX,\mathrm{rel}}_{0,n,\beta}.
\end{equation*}
\end{cor}

\subsubsection{Torus action}
In this paragraph, we assume we have a torus action on the family $\cX$ over $\aA^1$ leaving the fiber $X_0$ invariant.  Denote by $T$ the torus.  Then, we get a $T$-action on the moduli spaces $\cM_{\cX}$ and $\cM_{X_0}$, and the perfect obstruction theories $E_\cX$ on $\cM_{\cX}$ and $F_{X_0}$ on $\cM_{X_0}$ are also $T$-equivariant.

\begin{nt}
For a DM stack $\cY$ with a $T$-action, we denote by $\iota \colon \cY_T \hookrightarrow \cY$ the fixed locus.  For a $T$-equivariant perfect obstruction theory $E_\cY$ on $\cY$, we denote by $N^\vir_{T}$ the moving part of the dual of its restriction to the fixed locus $\cY_T$, and by $E_{T}$ the fixed part, which is a perfect obstruction theory on the fixed locus $\cY_T$.
\end{nt}

\begin{pro}[Localization formula, \textit{cf.} {\cite[Equation (8)]{GraberPandha}}]\label{GraberPandha}
Let $\cY$ be a DM stack with a $T$-action and a $T$-equivariant perfect obstruction theory $E \to L_\cY$.  Let $A^{T}_*(\cY)$ denote the $T$-equivariant Chow ring of\, $\cY$ and $\underline{t}$ denote the $T$-equivariant parameters.  Introduce the ring
$$A^T_*(\cY)_\loc := A_*^T(\cY) \otimes_{\QQ[\underline{t}]} \QQ\left[\underline{t}^{\pm1}\right]$$
obtained by inverting the equivariant parameters $\underline{t}$.  Then the virtual localization formula is
\begin{equation*}
[\cY,E] = \iota_* \(\cfrac{[\cY_T,E_T]}{e_T\(N^\vir_{T}\)}\),
\end{equation*}
where $e_T$ denotes the $T$-equivariant Euler class.
\end{pro}

We refer to \cite{Edidin} for a detailed construction of the equivariant Chow ring.

\begin{rem}
  In our situation, the fixed locus $\cX_T$ lies in the central fiber $X_0$, and we have $\cX_T = {X_0}_T$. Moreover, $\cX_T$
  is a smooth DM stack, and we denote it by $X_T$.
\end{rem}

\begin{thm}[Equivariant regular specialization theorem]\label{skid equiv}
Let $\cX$ be a $T$-equivariant regular family over $\aA^1$ whose fibers are embedded in a smooth proper DM stack $\cP$ and where the torus action leaves the central fiber invariant.  For every genus $g$, number of markings $n$ such that $2g-2+n>0$, isotropies $\underline{\rho} = (\rho_1, \dotsc, \rho_n)$ in $X_1$, curve class $\beta \in H_2(\cP)$, and insertions $\alpha_1, \dotsc, \alpha_n \in A^*(I\cX)$ with $\alpha_i \in A^*({\cX}_{\rho_i})$ admitting a $T$-equivariant lifting, we have
\begin{equation*}
\sum_{\substack{\beta_1 \in H_2(X_1) ~\textrm{with}\\\beta_1 = \beta \in H_2(\cP)}} \left[ \alpha_1,\dotsc,\alpha_n \right]^{X_1,\lambda_g}_{g,n,\beta_1}
 =  \lim_{\underline{t} \to 0} (-1)^g \int_{[{\cM_{X_T}},E_T]_{g,n,\beta}} \cfrac{\prod_{i=1}^n \mathrm{ev}_T^*(\alpha_i)}{e\(N_T^\vir\)},
\end{equation*}
where $\mathrm{ev}_T = \mathrm{ev}_{\cX} \circ j_0 \circ \iota_{{\cM_{X_T}}}$ and $r_T = r_{X_0} \circ \iota_{{\cM_{X_T}}}$, and $[{\cM_{X_T}},E_T]$ is the Gromov--Witten virtual fundamental cycle of the moduli space of stable maps to the smooth DM stack $X_T$.
\end{thm}

\begin{proof}
We start with the equality given in Theorem~\ref{skid}.  Then, we notice that the right-hand side is the non-equivariant limit of its well-defined $T$-equivariant enhancement.  Finally, we use the virtual localization formula to conclude; see Proposition~\ref{GraberPandha}.
\end{proof}

\begin{rem}
The regular specialization theorem and its equivariant version work as well with a regular family defined over an affine basis $\aA^m$. Precisely, we then have to multiply the virtual cycle by the $\supth{m}$ power of the Hodge class. Note that it is only interesting in genus zero, see Section~\ref{secinv}, as for positive genus the power of the Hodge class vanishes.
\end{rem}

\section{Equivariant quantum Lefschetz theorem}\label{QLsection}
In this section, we state an `equivariant quantum Lefschetz' theorem which will be useful for computations in the next section.

\subsection{$\boldsymbol{\CC^*}$-localization versus torus localization}
Let $\cX \hookrightarrow \cP$ be an embedding of smooth DM stacks equipped with a $\CC^*$-action.
We assume that
\begin{itemize}
\item the $\CC^*$-fixed loci of $\cX$ and of $\cP$ are equal;
\item $\cP$ is equipped with a torus action $T = (\CC^*)^N$, extending the $\CC^*$-action by an embedding $\CC^* \hookrightarrow T$;
\item the normal bundle of $\cX \hookrightarrow \cP$ is the pull-back of a $T$-equivariant vector bundle $\cN$ over $\cP$;
\item the vector bundle $\cN$ is convex up to two markings;\footnote{This notion is close to the so-called \textit{weak convexity} defined in the context of quasi-maps in \cite{Shoemaker}. However, it differs from their definition, mainly because we do not twist the vector bundle $\cN$ by one of the two markings here.} \textit{i.e.}~for every stable map $f \colon \cC \to \cP$, where $\cC$ is a smooth genus zero orbifold curve with at most two markings, we have $H^1(\cC,f^*\cN)=0$.
\end{itemize}

First, we look at the $\CC^*$-fixed loci of the moduli spaces of stable maps, and we find the following fibered diagram: 
\begin{center}
	\begin{tikzpicture}[scale=0.75]	
	\node (A') at (0,1.5) {$\cM(\cX)^{\CC^*}$};
	\node (B') at (3,1.5) {$\cM(\cP)^{\CC^*}$};
	\node (C') at (3,0) {$\cM(\cP)$.};
	\node (D') at (0,0) {$\cM(\cX)$};
	\draw[->] (D') -- (C');
	\draw[->] (A') -- (B');
	\draw[->] (A') -- (D');
	\draw[->] (B') -- (C');
	\draw[-] (1.5-0.1,0.75-0.1) -- (1.5-0.1,0.75+0.1) -- (1.5+0.1,0.75+0.1) -- (1.5+0.1,0.75-0.1) -- (1.5-0.1,0.75-0.1);
	\node[above] at (1.5,1.5){$j$};
	\node[right] at (3,0.75){$\widetilde{\iota}$};
	\node[left] at (0,0.75){$\iota$};
	\node[below] at (1.5,0){$\widetilde{j}$};
	\end{tikzpicture}
\end{center}
Writing $i \colon \cX \hookrightarrow \cP$, we have a $\CC^*$-equivariant short exact sequence
$$0 \lra T_{\cX} \lra i^*T_{\cP} \lra i^*\cN \lra 0,$$
which induces a distinguished triangle for the dual of the perfect obstruction theories of $\cM(\cX)$ and of $\cM(\cP)$
\begin{equation}\label{exactseqobstr}
R\pi_*f^*T_{\cX} \lra R\pi_*f^*T_{\cP} \lra R\pi_*f^*\cN \lra \(R\pi_*f^*T_{\cX}\)[1].
\end{equation}

The term $\cE := R\pi_*f^*\cN$, pulled back to $\cM(\cP)^{\CC^*}$, has a fixed and a moving part, which we denote, respectively, by $\cE_\mathrm{fix}$ and $\cE_\mathrm{mov}$.

\begin{pro}
The fixed part $\cE_\mathrm{fix}$ is a vector bundle over the fixed moduli space $\cM(\cP)^{\CC^*}$.
\end{pro}

\begin{proof}
Let $f \colon \cC \to \cP$ be a stable map belonging to $\cM(\cP)^{\CC^*}$.  We denote by $\rho \colon \cC \to C$ the coarse map.  It is enough to prove
$$H^1\(C,\rho_*f^* \cN\)^\mathrm{fix}=0.$$

Take the normalization $\nu \colon C^\nu \to C$ of the curves at all their nodes. We have
$$C^\nu = \bigsqcup_{i \in I} C^\mathrm{fix}_i \sqcup \bigsqcup_{j \in J} C^\mathrm{nf}_j,$$
where the superscripts refer, respectively, to fixed and non-fixed components of $C^\nu$ under the map $f$ and the $\CC^*$-action.  In particular, a non-fixed component is unstable and maps to a one-dimensional $\CC^*$-orbit.  By the normalization exact sequence, we obtain an exact sequence
$$\bigoplus_{\mathrm{nodes}} H^0\(\mathrm{node},{f^* \cN}_{|\mathrm{node}}\) \lra H^1(C,f^* \cN) \lra  H^1(C^\nu,\nu^*f^* \cN) \lra 0,$$
with
$$H^1(C^\nu,\nu^*f^* \cN) = \bigoplus_{i \in I} H^1\(C^\mathrm{fix}_i,\nu^*f^* \cN\) \oplus \bigoplus_{j \in J} H^1\(C^\mathrm{nf}_j,\nu^*f^* \cN\).$$
Since the normal bundle has a non-trivial $\CC^*$-action once restricted to the fixed locus of $\cX$ (or equivalently of $\cP$),  we then have
$$H^0\(\mathrm{node},{f^* \cN}_{|\mathrm{node}}\)^\mathrm{fix}=0 \quad \mathrm{and} \quad H^1\(C^\mathrm{fix}_i,\nu^* f^* \cN\)^\mathrm{fix}=0.$$
Therefore, it remains to see the vanishing of $H^1$ for non-fixed unstable curves $C_j^\mathrm{nf}$, $j \in J$.
The curve $C_j^\mathrm{nf}$ is isomorphic to $\PP^1$ with either one or two markings; hence $H^1(C^\mathrm{nf}_j,\nu^*f^* \cN)=0$ by our assumption of convexity up to two markings.
\end{proof}

Denote by $\left[\cM(\cP)^{\CC^*}\right]^\vir$ the virtual fundamental cycle obtained by the $\CC^*$-fixed part of the perfect obstruction theory $R\pi_*f^*T_{\cP}$.

\begin{pro}
In the Chow ring of $\cM(\cP)^{\CC^*}$, we have
$$j_*\left[\cM(\cX)^{\CC^*}\right]^\vir = e_{\CC^*}\(\cE_\mathrm{fix}\) \cdot \left[\cM(\cP)^{\CC^*}\right]^\vir.$$
Furthermore, in the localized equivariant Chow ring, we have
$$e_{\CC^*}(N^\vir_\iota)^{-1} = j^* \(\cfrac{e_{\CC^*}(\cE_\mathrm{mov})}{e_{\CC^*}\(N^\vir_{\widetilde{\iota}}\)}\).$$
\end{pro}

\begin{proof}
The results follow from the standard proof using convexity;  we recall  the main arguments.

The variety $\cX$ is the zero locus of a section of the vector bundle $\cN$ over the ambient space $\cP$.  This section induces a map $s$ from the moduli space of stable maps to $\cP$ to the direct image cone $\pi_*f^*\cN$; see \cite[Definition $2.1$]{Li}.  Since the moduli space $\cM(\cP)^{\CC^*}$ is fixed by the action of $\CC^*$,  it then maps to the fixed part of the direct image cone, that is, the vector bundle $\cE_\mathrm{fix}$.  Hence we have the fibered diagram
\begin{center}
	\begin{tikzpicture}[scale=0.75]	
	\node (A') at (0.5+6+0,1.5) {$\cM(\cX)^{\CC^*}$};
	\node (B') at (-0.5+1+6+4.5,1.5) {$\cM(\cP)^{\CC^*}$};
	\node (C') at (-0.5+1+6+4.5,0) {$\cE_\mathrm{fix}$,};
	\node (D') at (0.5+6+0,0) {$\cM(\cP)^{\CC^*}$};
	\draw[->] (D') -- (C');
	\draw[->] (A') -- (B');
	\draw[->] (A') -- (D');
	\draw[->] (B') -- (C');
	\draw[-] (1.5+6+1.25-0.1,0.75-0.1) -- (1.5+6+1.25-0.1,0.75+0.1) -- (1.5+6+1.25+0.1,0.75+0.1) -- (1.5+6+1.25+0.1,0.75-0.1) -- (1.5+6+1.25-0.1,0.75-0.1);
	\node[above] at (6.25+1.25+1.3,1.5){$j$};
	\node[right] at (6.5+2.5+2,0.75){$s$};
	\node[below] at (6.25+1.25+1.3,0){$0$};
	\end{tikzpicture}
\end{center}
where the bottom map is the embedding as the zero section.  The fixed part of the distinguished triangle \eqref{exactseqobstr} gives a compatibility datum of perfect obstruction theories for the fixed moduli spaces.  The functoriality of the virtual fundamental cycle gives
$$0^{!} \left[\cM(\cP)^{\CC^*}\right]^\vir = \left[\cM(\cX)^{\CC^*}\right]^\vir,$$
which is the desired result once we push forward via the map $j$ on both sides.  The second part of the statement follows from the moving part of the distinguished triangle \eqref{exactseqobstr}.
\end{proof}

It follows, by the virtual localization formula, that the $\CC^*$-equivariant virtual cycle satisfies
\begin{align}\label{equ}
\widetilde{j}_*\left[\cM(\cX)\right]^{\vir,\CC^*} & =  \widetilde{j}_*\iota_* \left( \cfrac{\left[\cM(\cX)^{\CC^*}\right]^\vir}{e_{\CC^*}\(N^\vir_\iota\)} \right) \nonumber \\
& =  \widetilde{\iota}_*j_* \left(\left[\cM(\cX)^{\CC^*}\right]^\vir \cdot j^* \(\cfrac{e_{\CC^*}(\cE_\mathrm{mov})}{e_{\CC^*}\(N^\vir_{\widetilde{\iota}}\)}\) \right) \nonumber \\
& =  \widetilde{\iota}_* \left(e_{\CC^*}\(\cE_\mathrm{fix}\) \cdot \left[\cM(\cP)^{\CC^*}\right]^\vir \cdot  \cfrac{e_{\CC^*}(\cE_\mathrm{mov})}{e_{\CC^*}\(N^\vir_{\widetilde{\iota}}\)} \right) \nonumber \\
& =  \widetilde{\iota}_* \left(e_{\CC^*}\(\cE\) \cdot \cfrac{\left[\cM(\cP)^{\CC^*}\right]^\vir}{e_{\CC^*}\(N^\vir_{\widetilde{\iota}}\)} \right),
\end{align}
where the equalities happen in the $\CC^*$-localized equivariant Chow ring $A_*^{\CC^*}(\cM(\cP))_\loc$.

\begin{rem}\label{defafterloc}
If it were defined, the right-hand side would equal
$$e_{\CC^*}\(R\pi_*f^*\cN\) \cdot \left[\cM(\cP)\right]^{\vir,\CC^*},$$
using the virtual localization formula, but it is not clear that the $\CC^*$-equivariant Euler class of $R\pi_*f^*\cN$ is defined in $A_*^{\CC^*}(\cM(\cP))_\loc$.  However, we say that $e_{\CC^*}\(R\pi_*f^*\cN\)$ is defined after localization\footnote{We can find this definition in \cite{Lho} for the formal quintic.} 
to mean that its pull-back to the fixed locus is defined.
\end{rem}

Now, we aim to extend the right-hand side of the equality to the action of a torus $T$.
We first consider the following general situation.

\begin{cons}\label{partdef}
Let us consider a DM stack $\cY$ with a $T$-action, and thus a $\CC^*$-action via an embedding $\CC^* \hookrightarrow T$.  We then have a pull-back map
$$\xi^* \colon A_*^{T}(\cY) \lra A_*^{\CC^*}(\cY)$$
coming from the map of algebraic stacks
\begin{equation*}
[\cY/\CC^*] \lra [\cY/T]
\end{equation*}
induced by the embedding $\CC^* \hookrightarrow T$.
More concretely, we express the $T$-parameters $t_1, \dotsc, t_N$ in terms of the $\CC^*$-parameter $t$ using the embedding $\CC^* \hookrightarrow T$.

The localized ring in which the (virtual) localization formula happens is defined as follows:
\begin{align*}
A_*^{T}(\cY)_\loc & :=  A_*^{T}(\cY) \otimes_{\QQ[t_1,\dotsc,t_N]} \QQ(t_1,\dotsc,t_N), \\
A_*^{\CC^*}(\cY)_\loc & :=  A_*^{\CC^*}(\cY) \otimes_{\QQ[t]} \QQ(t),
\end{align*}
where the rings of polynomials are indeed $A^T_*(\mathrm{pt})$ and $A^{\CC^*}_*(\mathrm{pt})$, and where we use the pull-back along $\cY \to \mathrm{pt}$.  Note that the inclusion of polynomials into their fraction field is injective, so  $A_*^{T}(\cY) \subset A_*^{T}(\cY)_\loc$ is a subring.

The pull-back map $\xi^*$ can almost be extended to the localized rings. The only issue is that the evaluation map
$$\ev \colon \QQ(t_1,\dotsc,t_N) \longdashrightarrow \QQ(t)$$
is only partially defined, as we see for instance on the element
$$\cfrac{1}{t_1^{a_2}-t_2^{a_1}}$$
when $t \mapsto (t_1=t^{a_1}, t_2=t^{a_2})$.  Nevertheless, when defined, this map is a ring map.  Precisely, if $F_1, F_2$ in $\QQ(t_1,\dotsc,t_N)$ satisfy that $\ev(F_1)$ and $\ev(F_2)$ are well defined, then $\ev(F_1 F_2)$ is well defined and
$$\ev(F_1)\ev(F_2)=\ev(F_1F_2).$$
As a consequence, we have a partially defined ring map
$$\xi^* \colon A_*^{T}(\cY)_\loc \longdashrightarrow A_*^{\CC^*}(\cY)_\loc$$
defined by the formula $\alpha \otimes F \mapsto \alpha \otimes \ev(F)$.  Furthermore, it satisfies some functorial properties.  Let $i \colon \cY_1 \to \cY_2$ be a $T$-equivariant map on DM stacks equipped with a $T$-action.  Then the push-forward map
$$i_* \colon A^T_*(\cY_1) \lra A^T_*(\cY_2)$$
extends to a push-forward map
$$i_* \colon A^T_*(\cY_1)_\loc \lra A^T_*(\cY_2)_\loc$$
by the formula $\alpha \otimes F \mapsto i_*(\alpha) \otimes F$.  Therefore, for all elements $x \in A^T_*(\cY_1)_\loc$ such that $\xi^*(x)$ is well defined, the element $\xi^*(i_*(x))$ is well defined and we have
$$\xi^*(i_*(x)) = i_*(\xi^*(x)).$$
We also have the similar property with pull-backs, although we will not use it.
\end{cons}

Let us get back to our situation.  We consider the $T$-fixed locus $\cM(\cP)^T \hookrightarrow \cM(\cP)$ of the moduli space. In particular, we have the inclusion $\widehat{\iota} \colon \cM(\cP)^T \hookrightarrow \cM(\cP)^{\CC^*}$.  We notice that the moduli space $\cM(\cP)^{\CC^*}$ is stable under the $T$-action from $\cM(\cP)$ and that the map $\widehat{\iota}$ is $T$-equivariant.  Moreover, we have a $T$-equivariant virtual cycle
$$\left[\cM(\cP)^{\CC^*}\right]^{\vir,T} \in A_*^T\(\cM(\cP)^{\CC^*}\)$$
and the equality $$\xi^*\(\left[\cM(\cP)^{\CC^*}\right]^{\vir,T}\) = \left[\cM(\cP)^{\CC^*}\right]^{\vir} \in A_*^{\CC^*}\(\cM(\cP)^{\CC^*}\).$$

By the virtual localization formula, we have
$$\left[\cM(\cP)^{\CC^*}\right]^{\vir,T} = \widehat{\iota}_* \(\cfrac{\left[\cM(\cP)^T\right]^\vir}{e_T\(N^\vir_{\widehat{\iota}}\)} \) \in A_*^T\(\cM(\cP)^{\CC^*}\)_\loc.$$
Furthermore, we have the following equality in K-theory on the space $\cM(\cP)^T$: 
\begin{equation}\label{virtnormbund}
N^\vir_{\widetilde{\iota} \circ \widehat{\iota}} = \widehat{\iota}^*N^\vir_{\widetilde{\iota}} + N^\vir_{\widehat{\iota}}.
\end{equation}
Indeed, let $\cF$ be the pull-back of $R\pi_*f^*T\cP$ to $\cM(\cP)^T$. By definition, the virtual normal bundle $N^\vir_{\widetilde{\iota} \circ \widehat{\iota}}$ is the $T$-moving part $\cF^\mathrm{mov}$, which decomposes as $\cF^{\mathrm{mov}} = \cF^{\mathrm{mov}}_\mathrm{fix} + \cF^{\mathrm{mov}}_\mathrm{mov}$, where the subscript denotes the $\CC^*$-fixed/moving part.  By definition, the virtual normal bundle $\widehat{\iota}^*N^\vir_{\widetilde{\iota}}$ is the $\CC^*$-moving part of $\cF$, \textit{i.e.}~$\cF^{\mathrm{mov}}_\mathrm{mov}$, since there is no $\CC^*$-moving $T$-fixed part in $\cF$.  It follows that the virtual normal bundle $N^\vir_{\widehat{\iota}}$ identifies with $\cF^{\mathrm{mov}}_\mathrm{fix}$.

\begin{rem}
The virtual normal bundle $N^\vir_{\widetilde{\iota}}$ is defined on $\cM(\cP)^{\CC^*}$, and we have an equality
$$\xi^*\(e_T\(N^\vir_{\widetilde{\iota}}\)^{-1}\) = e_{\CC^*}\(N^\vir_{\widetilde{\iota}}\)^{-1} \in A_*^{\CC^*}(\cM(\cP)^{\CC^*})_\loc,$$
which is well defined because the right-hand side is well defined.  We  have also seen the $\CC^*$-decomposition $\cE = \cE_\mathrm{fix} + \cE_\mathrm{mov}$ over $\cM(\cP)^{\CC^*}$ with $\cE_\mathrm{fix}$ being a $T$-equivariant vector bundle.  Indeed, the vector bundle $\cN$ over $\cP$ is $T$-equivariant; thus, so are $\cE$ and $\cE_\mathrm{fix}$.  As a consequence, the equality
$$\xi^*\(e_T(\cE)\) = e_{\CC^*}(\cE) \in A_*^{\CC^*}(\cM(\cP)^{\CC^*})_\loc$$
is well defined because the right-hand side is well defined.
\end{rem}

\begin{pro}\label{proploc}
Consider the well-defined class
$$C_T :=\widehat{\iota}^*\(e_{T}\(\cE\)\) \cdot \cfrac{\left[\cM(\cP)^{T}\right]^\vir}{e_{T}\(N^\vir_{\widetilde{\iota} \circ \widehat{\iota}}\)} \in A_*^{T}(\cM(\cP)^T)_\loc.$$
Its push-forward under the inclusion $\widehat{\iota}$ equals
$$\widehat{\iota}_*\(C_T\) = e_T\(\cE\) \cdot \cfrac{\left[\cM(\cP)^{\CC^*}\right]^{\vir,T}}{e_T\(N^\vir_{\widetilde{\iota}}\)} \in A^T_*\(\cM(\cP)^{\CC^*}\)_\loc.$$
In particular, we have
$$\xi^*\(\widehat{\iota}_*\(C_T\)\) = e_{\CC^*}\(\cE\) \cdot \cfrac{\left[\cM(\cP)^{\CC^*}\right]^\vir}{e_{\CC^*}\(N^\vir_{\widetilde{\iota}}\)} \in A^{\CC^*}_*\(\cM(\cP)^{\CC^*}\)_\loc.$$
\end{pro}

\begin{proof}
By the virtual localization above and Equation \eqref{virtnormbund}, we have
\begin{align*}
\widehat{\iota}_*\(C_T\) & =  e_{T}\(\cE\) \cdot \widehat{\iota}_* \(\cfrac{\left[\cM(\cP)^{T}\right]^\vir}{e_{T}\(N^\vir_{\widetilde{\iota} \circ \widehat{\iota}}\)} \) \\
& =  e_{T}\(\cE\) \cdot \widehat{\iota}_* \(\cfrac{\left[\cM(\cP)^{T}\right]^\vir}{\widehat{\iota}^*\(e_{T}\(N^\vir_{\widetilde{\iota}}\)\) \cdot e_{T}\(N^\vir_{\widehat{\iota}}\)} \) \\
& =  \cfrac{e_{T}\(\cE\)}{e_{T}\(N^\vir_{\widetilde{\iota}}\)} \cdot \widehat{\iota}_* \(\cfrac{\left[\cM(\cP)^{T}\right]^\vir}{e_{T}\(N^\vir_{\widehat{\iota}}\)} \) \\
& =  \cfrac{e_{T}\(\cE\) \cdot \left[\cM(\cP)^{\CC^*}\right]^{\vir,T}}{e_{T}\(N^\vir_{\widetilde{\iota}}\)}. \\
\end{align*}
The last sentence follows, as we already discussed in Construction~\ref{partdef}, from the property that $\xi^*$ is a ring map when it is defined, and we have seen that $\xi^*$ is well defined on the three factors on the right-hand side.
\end{proof}

\subsection{Equivariant quantum Lefschetz formula}
Summarizing our discussion, we obtain the following.

\begin{thm}[Equivariant quantum Lefschetz formula]\label{quantum Lefschetz}
Let $\cX \hookrightarrow \cP$ be a $\CC^*$-equivariant embedding of smooth DM stacks satisfying the assumptions listed at the beginning of this section.  Then we have
$$\widetilde{j}_*\left[\cM(\cX)\right]^{\vir,\CC^*}=\xi^* \( e_T\(R\pi_*f^*\cN\) \cdot \left[\cM(\cP)\right]^{\vir,T} \) \in A^{\CC^*}_*(\cM(\cP))_\loc,$$
where $\widetilde{j}$ is the embedding of moduli spaces and $\xi^*$ is the specialization of $T$-equivariant parameters into the $\CC^*$-equivariant parameter.
Here, the $T$-equivariant Euler class $e_T (R \pi_*f^*\cN)$ is defined after localization; see Remark~\ref{defafterloc}.
\end{thm}

\begin{proof}
Using previous equalities, we get
\begin{align*}
\widetilde{j}_*\left[\cM(\cX)\right]^{\vir,\CC^*}
& =  \widetilde{\iota}_* \left(e_{\CC^*}\(\cE\) \cdot \cfrac{\left[\cM(\cP)^{\CC^*}\right]^\vir}{e_{\CC^*}\(N^\vir_{\widetilde{\iota}}\)} \right) \\
& =  \widetilde{\iota}_*\(\xi^*\(\widehat{\iota}_*\(C_T\)\)\) \\
& =  \xi^*\(\widetilde{\iota}_*\(\widehat{\iota}_*\(C_T\)\)\) \\
& =  \xi^* \widetilde{\iota}_* \widehat{\iota}_* \(  \widehat{\iota}^*\(e_{T}\(\cE\)\) \cdot \cfrac{\left[\cM(\cP)^{T}\right]^\vir}{e_{T}\(N^\vir_{\widetilde{\iota} \circ \widehat{\iota}}\)} \),
\end{align*}
where we first use Equation \eqref{equ}, then Proposition~\ref{proploc}, then the last property in Construction~\ref{partdef}, and finally the definition of $C_T$.  Following Remark~\ref{defafterloc}, the meaning of `defined after localization' is precisely
\begin{align*}\pushQED{\qed}
\xi^* \( e_T\(R\pi_*f^*\cN\) \cdot \left[\cM(\cP)\right]^{\vir,T} \) & =  \xi^* \( e_{T}\(R\pi_*f^*\cN\) \cdot \widetilde{\iota}_* \widehat{\iota}_* \(\cfrac{\left[\cM(\cP)^{T}\right]^\vir}{e_{T}\(N^\vir_{\widetilde{\iota} \circ \widehat{\iota}}\)}  \) \)\\
& =  \xi^* \widetilde{\iota}_* \widehat{\iota}_* \(  \widehat{\iota}^*\(e_{T}\(\cE\)\) \cdot \cfrac{\left[\cM(\cP)^{T}\right]^\vir}{e_{T}\(N^\vir_{\widetilde{\iota} \circ \widehat{\iota}}\)} \).
\qedhere \popQED
\end{align*}
\renewcommand{\qed}{}    
\end{proof}

Finally, we summarize Sections~\ref{sec1} and~\ref{QLsection} in the following theorem, whose set-up is the following: 

Let $\cX \hookrightarrow \cP$ be a $\CC^*$-equivariant embedding of regular $\aA^1$-families.  We assume that
\begin{itemize}
	\item the fixed loci of $\cX$ and of $\cP$ are equal; 
	\item the ambient space $\cP$ carries a $T:=\(\CC^*\)^N$-action; \textit{e.g.}~it is a toric DM stack, extending the $\CC^*$-action via an embedding $\CC^* \hookrightarrow T$; 
	\item the normal bundle $\cN$ of $\cX \hookrightarrow \cP$ is a pull-back from a $T$-equivariant vector bundle over $\cP$, and is convex up to two markings; \textit{e.g.}~it satisfies some positivity condition.
\end{itemize}

Let $X$ be a generic smooth fiber of $\cX$.  Fix a genus $g$, a number of markings $n$ such that $2g-2+n>0$, isotropies $\underline{\rho} = (\rho_1, \dotsc, \rho_n)$ in $X$, a curve class $\beta \in H_2(\cP)$, and ambient insertions $\alpha_1, \dotsc, \alpha_n \in A^*(I\cP)$ with $\alpha_i \in A^*({\cP}_{\rho_i})$ admitting $T$-equivariant liftings.  We set $\alpha := \prod_{i=1}^n \ev_i^*(\alpha_i)$ to be the product of insertions.

\begin{thm}\label{big theo}
Under the set-up and assumptions listed just above, we have the following equality in the Chow ring of\, $\overline{\cM}_{g,n}$: 
$$\left[\cM_{g,\underline{\rho}}(\cP,\beta)\right]^{\vir,T} \cdot e_T (R \pi_*f^*\cN) \cdot \alpha \xrightarrow[t \to 0]{} e\(\mathbb{E}^\vee\) \cdot \left[\cM_{g,\underline{\rho}}(X,\beta)\right]^\vir \cdot \alpha,$$
where on the right-hand side we take the sum over curve classes $\beta_1 \in H_2(X)$ such that $\beta_1 = \beta \in H_2(\cP)$.
Precisely, the class $e_T (R \pi_*f^*\cN)$ is only defined after localization, so we first apply the virtual localization formula to the left-hand side, then we compute it in $A^T_*(\overline{\cM}_{g,n} \times \aA^1)$ as a formal series in the $T$-equivariant parameters and their inverses, then we specialize them to a single variable $t$ using $\CC^* \hookrightarrow T$ and obtain a well-defined polynomial in $t$, and finally we take the constant coefficient and pull it back from $A_*(\overline{\cM}_{g,n} \times \aA^1)$ to $A_*(\overline{\cM}_{g,n})$.

Furthermore, if the regular family $\cP$ is trivial, \textit{i.e.}~$\cP = P \times \aA^1$ for some $T$-equivariant smooth DM stack $P$, then the formula simplifies as
$$e_T\(\mathbb{E}^\vee\) \cdot \left[\cM_{g,\underline{\rho}}(P,\beta)\right]^{\vir,T} \cdot e_T (R \pi_*f^*\cN) \cdot \alpha \xrightarrow[t \to 0]{} e\(\mathbb{E}^\vee\) \cdot \left[\cM_{g,\underline{\rho}}(X,\beta)\right]^\vir \cdot \alpha,$$
with the same meaning as above, except that we work directly in $A^T_*(\overline{\cM}_{g,n})$ instead of $A^T_*(\overline{\cM}_{g,n} \times \aA^1)$.
\end{thm}

\begin{proof}
By Theorem~\ref{quantum Lefschetz}, we obtain
$$\widetilde{j}_*\left[\cM(\cX)\right]^{\vir,\CC^*}=\xi^* \( e_T\(R\pi_*f^*\cN\) \cdot \left[\cM(\cP)\right]^{\vir,T} \) \in A^{\CC^*}_*(\cM(\cP))_\loc,$$
where $\widetilde{j}$ is the embedding of moduli spaces and $\xi^*$ is the specialization of $T$-equivariant parameters into the $\CC^*$-equivariant parameter, corresponding to the specialization induced by $\CC^* \hookrightarrow T$.  We recall that the $T$-equivariant Euler class $e_T (R \pi_*f^*\cN)$ is only defined after localization; see Remark~\ref{defafterloc}.

By Proposition~\ref{skid 2} and Lemma~\ref{Hodge}, we get
$$e\(\mathbb{E}^\vee\) \cdot q_{1*} \( \left[\cM(X)\right]^{\vir} \cdot \alpha \) = 1^*q_* \( \left[\cM(\cX)\right]^{\vir} \cdot \alpha \),$$
where $q \colon \cM(\cX) \to \aA^1 \times \overline{\cM}_{g,n}$, $q_1 \colon \cM(\cX_{s=1}) \to 1 \times \overline{\cM}_{g,n}$, and $1 \colon 1 \times \overline{\cM}_{g,n} \to \aA^1 \times \overline{\cM}_{g,n}$ form a fibered diagram; see the beginning of Section~\ref{secpot}.
Using the $T$-equivariant map $0 \colon 0 \times \overline{\cM}_{g,n} \to \aA^1 \times \overline{\cM}_{g,n}$,  in $A_*(\overline{\cM}_{g,n})$ we have the equality
$$1^*q_* \( \left[\cM(\cX)\right]^{\vir} \cdot \alpha \) = 0^*q_* \( \left[\cM(\cX)\right]^{\vir} \cdot \alpha \),$$
and the right-hand side equals the non-equivariant limit of
$0^*q_* \( \left[\cM(\cX)\right]^{\vir,\CC^*} \cdot \alpha \)$.
We use the notation $\widetilde{q} \colon \cM(\cP) \to \aA^1 \times \overline{\cM}_{g,n}$; we obtain $e(\mathbb{E}^\vee) \cdot q_{1*} \( \left[\cM(X)\right]^{\vir} \cdot \alpha \)$ as the non-equivariant limit of
\begin{equation}\label{limit}
0^*\widetilde{q}_*\(\xi^* \( e_T\(R\pi_*f^*\cN\) \cdot \left[\cM(\cP)\right]^{\vir,T} \) \cdot \alpha \),
\end{equation}
where we recall that $\alpha$ consists of ambient insertions, so  it is a pull-back from $\cM(\cP)$.
Since there exists a $T$-equivariant lift of $\alpha$, it can be written as $\xi^*(\alpha)$, and by ring properties of the push-forward $\xi^*$ (see the end of the proof of Proposition~\ref{proploc}), and by the commutativity of push-forwards, Equation \eqref{limit} equals
\begin{equation}\label{limit2}
0^*\xi^*\( \widetilde{q}_*\( e_T\(R\pi_*f^*\cN\) \cdot \left[\cM(\cP)\right]^{\vir,T} \cdot \alpha \)  \),
\end{equation}
yielding the first part of the statement.

For the second part, we assume that $\cP = P \times \aA^1$, so that the normal bundle of $P \subset \cP$ is the trivial line bundle $\cO$ with a non-trivial $T$-action.  Over the fixed moduli space $\cM(\cP)^T$, the term $R\pi_*\cO =[\cO \to \mathbb{E}^\vee]$ then has no fixed part, and we get
$$\left[\cM(\cP)\right]^{\vir,T} = e_T\(R\pi_*\cO\)^{-1} \cdot \left[\cM(P)\right]^{\vir,T}$$
after localization.  Recall that Equation \eqref{limit2} is meant to be after localization as well, so  the map $\widetilde{q}$ is rather a map from the $T$-fixed moduli space $\cM(\cP)^T$, and thus factors through the map $0$, as the fixed moduli space lies above the central fiber. Thus we write $\widetilde{q} = 0 \circ \overline{q}$ with $\overline{q} \colon \cM(P) \to \overline{\cM}_{g,n}$.  Hence, Equation~\eqref{limit2} becomes
\begin{align*}
&  0^*\xi^*\( 0_* \overline{q}_*\( e_T\(R\pi_*f^*\cN\) \cdot e_T\(R\pi_*\cO\)^{-1} \cdot \left[\cM(P)\right]^{\vir,T} \cdot \alpha \)  \) \\
& =\enspace  0^*0_* \xi^*\( \overline{q}_*\( e_T\(R\pi_*f^*\cN\) \cdot e_T\(R\pi_*\cO\)^{-1} \cdot \left[\cM(P)\right]^{\vir,T} \cdot \alpha \)  \) \\
& = \enspace e_{\CC^*}(\cN_0) \cdot \xi^*\( \overline{q}_*\( e_T\(R\pi_*f^*\cN\) \cdot e_T\(R\pi_*\cO\)^{-1} \cdot \left[\cM(P)\right]^{\vir,T} \cdot \alpha \)  \) \\
&= \enspace \xi^*\( e_T(\cN_0) \cdot \overline{q}_*\( e_T\(R\pi_*f^*\cN\) \cdot e_T\(R\pi_*\cO\)^{-1} \cdot \left[\cM(P)\right]^{\vir,T} \cdot \alpha \)  \) \\
& =\enspace  \xi^*\( \cfrac{e_T(\cN_0)}{e_T\(\cO\)} \cdot \overline{q}_*\( e_T\(R\pi_*f^*\cN\) \cdot e_T\(\mathbb{E}^\vee\) \cdot \left[\cM(P)\right]^{\vir,T} \cdot \alpha \)  \), \\
\end{align*}
where $\cN_0$ is the normal bundle of $0 \colon 0 \times \overline{\cM}_{g,n} \to \aA^1 \times \overline{\cM}_{g,n}$ and therefore equals the trivial bundle $\cO$ with the same non-trivial $T$-action as before.  As a consequence, we get the desired left-hand side
\begin{equation*}\pushQED{\qed}
\xi^*\(\overline{q}_*\( e_T\(R\pi_*f^*\cN\) \cdot e_T\(\mathbb{E}^\vee\) \cdot \left[\cM(P)\right]^{\vir,T} \cdot \alpha \)  \).
\qedhere \popQED
\end{equation*}
\renewcommand{\qed}{}     
\end{proof}

\section{Smooth hypersurfaces in weighted projective spaces}

\subsection{Hodge--Gromov--Witten theory for chain polynomials}
Let $w_1, \dotsc, w_N$ be positive integers, and denote by $\PP(\underline{w})=\PP(w_1,\dotsc,w_N)$ the weighted projective space given by these weights.  In this subsection, we assume the chain-type arithmetic condition: there exist positive integers $a_1, \dotsc, a_N$ and $d$ such that
\begin{equation}\label{arithcondchain}
a_j w_j + w_{j+1} = d \quad \textrm{for $j<N$ and } a_N w_N = d.
\end{equation}
In particular, it gives the existence of a smooth (orbifold) hypersurface $X$ of degree $d$ in $\PP(\underline{w})$.  Precisely, one such example is the vanishing locus of the chain polynomial
$$x_1^{a_1}x_2+\dotsb+x_{N-1}^{a_{N-1}}x_N+x_N^{a_N}.$$
Moreover, since Gromov--Witten theory is invariant under smooth deformations, we can refer to this example for our computations.

The weighted projective space $\PP(\underline{w})$ carries the action of a torus $T=\(\CC^*\)^N$.  We denote the equivariant parameters by $\underline{t}=(t_1,\dotsc,t_N)$.  For any integer $d \in \ZZ$ and any character $\chi \in \Hom(T,\CC)$, there is a $T$-equivariant line bundle $\cO_\chi(d)$.

\begin{rem}\label{weightO(d)}
In Theorem~\ref{HGWchain}, we take the trivial character $\chi$ on $\cO(1)$ and then take its $\supth{d}$ power. This means that $\cO(d)$ has weight $-\tfrac{d t_j}{w_j}$ in the affine chart $x_j=1$.  We refer to Remark~\ref{afterloc} below for another description of the action used on the line bundle $\cO(d)$.
\end{rem}

Denote the $T$-equivariant virtual fundamental cycle by
\begin{equation*}
[\cM(\PP(\underline{w}))]^{\vir,T} \in A_*^T(\cM(\PP(\underline{w}))),
\end{equation*}
so that its non-equivariant limit $\underline{t} \to 0$ gives back the virtual fundamental cycle.  Moreover, the derived object $R\pi_*f^*\cO_\chi(d)$, where $\pi$ is the projection map from the universal curve and $f$ is the universal stable map, is also $T$-equivariant.  Unfortunately, the expression
$$e_T(R\pi_*f^*\cO_\chi(d)) \cdot [\cM(\PP(\underline{w}))]^{\vir,T} \in A_*^T(\cM(\PP(\underline{w})))_\loc,$$
which is defined after localization, does not admit a non-equivariant limit $\underline{t} \to 0$ unless the convexity condition holds and thus $R\pi_*f^*\cO(d) = \pi_*f^*\cO(d)$ is a vector bundle.

\begin{rem}
Convexity holds in genus zero under the Gorenstein condition: $w_j | d$ for all $j$.  In that case, the non-equivariant limit $\underline{t} \to 0$ gives back the virtual cycle $[\cM(X)]^\vir$ of the moduli space of stable maps to a smooth degree $d$ hypersurface $X \subset \PP(\underline{w})$.
\end{rem}

In Theorem~\ref{HGWchain}, we overcome the difficulty of non-convexity with the help of the Hodge bundle $\mathbb{E}$.  First, we pull it back to the moduli space of stable maps to $\PP(\underline{w})$, and then we endow it with the following $T$-action: one rescales fibers of $\mathbb{E}$ by $t_N^{a_N}$.

Finally, we use insertions of ambient cohomology classes of $X$, \textit{i.e.}~which are pulled back from the ambient space $\PP(\underline{w})$.  These classes are naturally expressed in terms of hyperplane classes, so they admit $T$-invariant representatives.

As an application of our regular specialization theorem, we prove the following.

\begin{thm}[Hodge--Gromov--Witten theory of chain hypersurfaces]\label{HGWchain}
We assume condition \eqref{arithcondchain}, and we fix $g, n \in \NN$ such that $2g-2+n>0$, $\beta \in \NN$, and isotropies $\underline{\rho} = (\rho_1, \dotsc, \rho_n)$ in $\PP(\underline{w})$.  Let $X \subset \PP(\underline{w})$ be a smooth hypersurface of degree $d$ and $\alpha_1,\dotsc,\alpha_n$ be ambient cohomology classes on $X$, \textit{i.e.}~pulled back from $\PP(\underline{w})$.  We set $\alpha := \prod_{i=1}^n \ev_i^*(\alpha_i)$ to be the product of insertions.  Then we have the following equality in the Chow ring of\, $\overline{\cM}_{g,n}$: 
$$e_T\(\mathbb{E}^\vee\) \cdot  \left[\cM_{g,\underline{\rho}}(\PP(\underline{w}),\beta)\right]^{\vir,T} \cdot e_T (R \pi_*f^*\cO(d)) \cdot \alpha \xrightarrow[t \to 0]{} e\(\mathbb{E}^\vee\) \cdot \left[\cM_{g,\underline{\rho}}(X,\beta)\right]^\vir \cdot \alpha.$$
Precisely, the class $e_T (R \pi_*f^*\cO(d))$ is only defined after localization, so we first apply the virtual localization formula to the left-hand side, then we compute it in $A^T_*(\overline{\cM}_{g,n})$ as a formal series in the $T$-equivariant parameters and their inverses, then we specialize them to $$t_{j+1} = (-a_1) \dotsm (-a_j) t$$
for all $1 \leq j \leq N$ and obtain a well-defined polynomial in $t$, and finally  we take the constant coefficient.
\end{thm}

\begin{rem}\label{afterloc}
The specialization of $T$-equivariant parameters in terms of a single variable $t$ can be rephrased as an embedding $\CC^* \hookrightarrow T$.  Then by Equation \eqref{arithcondchain}, there is a $\CC^*$-invariant (singular) hypersurface of degree $d$
$$X_0 = \left\lbrace x_1^{a_1}x_2+\dotsb+x_{N-1}^{a_{N-1}}x_N = 0 \right\rbrace \subset \PP(\underline{w}),$$
and the line bundle $\cO(d)$ in Theorem~\ref{HGWchain} is its normal line bundle.  Therefore, it comes with a $\CC^*$-action.  To be more precise, look at the weights on fibers over the fixed locus, which consists of all coordinate points in $\PP(\underline{w})$.  At the point $(0,\dotsc,x_j=1,\dotsc,0) \in \PP(\underline{w})$, the $\CC^*$-action has weight $-\tfrac{d t_j}{w_j}$, as was announced in Remark~\ref{weightO(d)}.
\end{rem}

\begin{rem}
Theorem~\ref{HGWchain} yields an explicit formula for Hodge--Gromov--Witten invariants of $X$ as a sum over dual graphs.  Indeed, such a formula is known for weighted projective spaces; see \textit{e.g.}~\cite{Liu}.
\end{rem}

\begin{proof}
As Gromov--Witten theory is invariant under smooth deformations, we can take the degree $d$ hypersurface $X$ to be the zero locus of the chain polynomial
$$P = x_1^{a_1}x_2+\dotsb+x_{N-1}^{a_{N-1}}x_N+x_N^{a_N}.$$
Define the regular family over $\aA^1$
$$\cX = \left\lbrace x_1^{a_1}x_2+\dotsb+x_{N-1}^{a_{N-1}}x_N+x_N^{a_N}s=0 \right\rbrace \subset \PP(w_1,\dotsc,w_N) \times \aA^1 =: \cP.$$
It is endowed with a $\CC^*$-action with weight $p_j$ on $x_j$ and $p_{N+1}$ on $s$ satisfying $p_1=1$ and
$p_{j+1} = (-a_1) \dotsm (-a_j)$
for $1 \leq j \leq N$.
Moreover, the fiber $X_1$ at $s=1$ equals the smooth hypersurface $X$, and the fiber $X_0$ at $s=0$ has exactly one singular point $\mathrm{Sing}(X_0) = (0,\dotsc,0,1) \in \PP(w_1,\dotsc,w_N)$.
It is then enough to check the assumptions of Theorem~\ref{big theo}.

The $\CC^*$-fixed loci for $\cX$ and for $\cP$ are the same; \textit{i.e.}~they are given by all $N$ coordinate points in the central fiber $s=0$.  The DM stack $\cP$ carries a $T=\(\CC^*\)^N$-action, where the action of $T$ on $\aA^1$ is the multiplication by $t_N^{-a_N}$.  The normal bundle of $\cX \hookrightarrow \cP$ is the pull-back of the $T$-equivariant line bundle $\cO(d)$ on $\PP(\underline{w})$, with the trivial character, as explained in Remarks~\ref{weightO(d)} and~\ref{afterloc}.  It remains to prove convexity up to two markings.

Let $f \colon \cC \to \cP$ be a stable map, where $\cC$ is a non-contracted smooth genus zero orbifold curve with two markings $\sigma_1$ and $\sigma_2$ (the cases with one or zero markings are similar).  Then $f^* \cO(d)$ is a line bundle over $\cC$ and can be written as
$$f^* \cO(d) = \cO(m + r_1 \sigma_1+r_2 \sigma_2),$$
with $0 \leq r_1, r_2 <1$ being the monodromies at the markings and $m \in \ZZ$. Moreover, we have the relation
$$m + r_1+r_2 = d \cdot \deg(f) \geq 0.$$
Hence, we have $m>-2$, so $m \geq -1$.
Therefore, denoting by $C=\PP^1$ the coarse curve of $\cC$, we have $H^1(\cC,f^* \cO(d)) = H^1(C,\cO(m)) = 0$.
\end{proof}

As a special case of Theorem~\ref{HGWchain}, we obtain a full genus zero computation of the Gromov--Witten theory of chain hypersurfaces with ambient insertions, using the simple fact that the Hodge class equals $1$ in genus zero.

\begin{cor}[Genus zero Gromov--Witten theory of chain hypersurfaces]\label{0HGWchain}
In the notation of Theorem~\ref{HGWchain} and under its assumptions, but with $g=0$, we have the following equality in the Chow ring of\, $\overline{\cM}_{0,n}$: 
$$\left[\cM_{0,\underline{\rho}}\(\PP(\underline{w}),\beta\)\right]^{\vir,T} \cdot e_T (R \pi_*f^*\cO(d)) \cdot \alpha \xrightarrow[t \to 0]{} \left[\cM_{0,\underline{\rho}}(X,\beta)\right]^\vir \cdot \alpha.$$
\end{cor}

\subsection{Hodge--Gromov--Witten theory for loop polynomials}
In this subsection, we assume the loop-type arithmetic condition: there exist positive integers $a_1, \dotsc, a_N$ and~$d$ such that
\begin{equation}\label{arithcondloop}
a_j w_j + w_{j+1} = d  \textrm{ for $j<N$}\quad \text{and}\quad a_N w_N + w_1 = d.
\end{equation}
In particular, this gives the existence of a smooth (orbifold) hypersurface $X$ of degree $d$ in $\PP(\underline{w})$.
Precisely, one such example is the vanishing locus of the loop polynomial
$$x_1^{a_1}x_2+\dotsb+x_{N-1}^{a_{N-1}}x_N+x_N^{a_N}x_1.$$ Moreover, since Gromov--Witten theory is invariant under smooth deformations, we can refer to this example for our computations.

\begin{rem}\label{weightO(d)2}
In Theorem~\ref{HGWloop}, we take the trivial character $\chi$ on $\cO(1)$ and then take its $\supth{d}$ power. This means that $\cO(d)$ has weight $-\tfrac{d t_j}{w_j}$ in the affine chart $x_j=1$.  We refer to Remark~\ref{afterloc2} below for another description of the action used on the line bundle $\cO(d)$.
\end{rem}

As in Theorem~\ref{HGWchain}, we overcome the difficulty of non-convexity with the help of the Hodge bundle $\mathbb{E}$, but we change the $T$-action to the following: one rescales fibers of $\mathbb{E}$ by $t_N^{a_N}t_1$.  Furthermore, we change the ambient space by performing a weighted blow-up.

Define a polynomial $Q$ by
$$Q(x_1, \dotsc, x_{N-1},s) = x_1^{a_1}x_2+\dotsb+x_{N-1}^{a_{N-1}}+x_1s.$$ Since $Q$ is a chain polynomial, it is quasi-homogeneous with some positive weights $b_1,\dotsc,b_{N-1},b_s$ and degree $\delta \in \NN^*$.

Let $\widetilde{\cP}$ be the weighted blow-up of $\PP(\underline{w}) \times \aA^1=:\cP$ at the point $((0,\dotsc,0,1),s=0)$ with weights $b_1,\dotsc,b_{N-1}$ on the variables $x_1, \dotsc, x_{N-1}$ in the chart $x_N=1$ and weight $b_s$ on the variable $s \in \aA^1$.  We refer to \cite{AbramWBl} for the construction of the weighted blow-up.

Let $T = \(\CC^*\)^{N+1}$ be the natural torus action on $\cP$.  Since the base locus $((0,\dotsc,0,1),s=0)$ is fixed under the torus $T$,  the space $\widetilde{\cP}$ then also carries a $T$-action; it is even a toric DM stack.  Moreover, the line bundle $\cO(d)$ defined in Remark~\ref{weightO(d)2} pulls back to a $T$-equivariant line bundle on $\widetilde{\cP}$, which we again denote by $\cO(d)$.  Let $E$ be the exceptional divisor of $\widetilde{\cP} \to \cP$. The line bundle $\cO(d-E)$ is also $T$-equivariant.

Furthermore, over $\aA^1-0$, we have $\widetilde{\cP} \simeq \cP$, so  we can embed the hypersurface $X \subset \PP(\underline{w})$ defined by the loop polynomial in the fiber over $s=1$, yielding
$$X \longhookrightarrow \widetilde{\cP}.$$
Let $\alpha \in H^*(X)$ be a cohomology class which is a pull-back from $\PP(\underline{w})$. It can be represented by a $T$\nobreakdash-equivariant cycle in $\PP(\underline{w})$ which does not contain the point $(0,\dotsc,0,1)$.
Therefore, we can view it as a $T$-equivariant cohomology class on $\widetilde{\cP}$, such that the pull-back to $X$ of its non-equivariant limit equals $\alpha$.

A curve class in $\PP(\underline{w})$ is a non-negative multiple of a line, which we can choose to avoid the point $(0,\dotsc,0,1)$ so that it gives a curve class in $\widetilde{\cP}$. We only consider these curve classes in the following statement.  Moreover, we assume condition \eqref{arithcondloop}, and we fix $g, n \in \NN$ such that $2g-2+n>0$, $\beta \in \NN$, and isotropies $\underline{\rho} = (\rho_1, \dotsc, \rho_n)$ in $\PP(\underline{w})$.  Let $X \subset \PP(\underline{w})$ be a smooth hypersurface of degree $d$ and $\alpha_1,\dotsc,\alpha_n$ be ambient cohomology classes on $X$, \textit{i.e.}~pulled back from $\PP(\underline{w})$.  We set $\alpha := \prod_{i=1}^n \ev_i^*(\alpha_i)$ to be the product of insertions.

\begin{thm}[Hodge--Gromov--Witten theory of loop hypersurfaces]\label{HGWloop}
Under the set-up and assumptions listed just above, we have the following equality in the Chow ring of\, $\overline{\cM}_{g,n}$: 
$$\left[\cM_{g,\underline{\rho}}\(\widetilde{\cP},\beta\)\right]^{\vir,T} \cdot e_T (R \pi_*f^*\cO(d-E)) \cdot \alpha \xrightarrow[t \to 0]{} e\(\mathbb{E}^\vee\) \cdot \left[\cM_{g,\underline{\rho}}(X,\beta)\right]^\vir \cdot \alpha.$$
Precisely, the class $e_T (R \pi_*f^*\cO(d-E))$ is only defined after localization, so we first apply the virtual localization formula to the left-hand side, then we compute it in $A^T_*(\overline{\cM}_{g,n} \times \aA^1)$ as a formal series in the $T$-equivariant parameters and their inverses, then we specialize them to
$$t_{N+1} = \((-a_1)\dotsm (-a_N) -1\) t ~, \quad t_{j+1} = (-a_1) \dotsm (-a_j) t$$
for all $1 \leq j \leq N-1$ and obtain a well-defined polynomial in $t$, and finally we take the constant coefficient and pull it back from $A_*(\overline{\cM}_{g,n} \times \aA^1)$ to $A_*(\overline{\cM}_{g,n})$.
\end{thm}

\begin{rem}\label{afterloc2}
The specialization of the $N$ first $T$-equivariant parameters in terms of a single variable $t$ is the same as in Theorem~\ref{HGWchain}, but the last one is different; \textit{i.e.}~the Hodge bundle is rescaled with weight $(-a_1) \dotsm (-a_N)-1$ instead of weight $(-a_1) \dotsm (-a_N)$.  By Equation \eqref{arithcondloop}, there is a $\CC^*$-invariant (singular) hypersurface of degree $d$
$$X_0 = \left\lbrace x_1^{a_1}x_2+\dotsb+x_{N-1}^{a_{N-1}}x_N = 0 \right\rbrace \subset \PP(\underline{w}),$$
and the line bundle $\cO(d)$ in Theorem~\ref{HGWloop} is its normal line bundle, with the same $\CC^*$-action as in Theorem~\ref{HGWchain}.
The line bundle $\cO(d-E)$ is also a normal bundle, as we see in the proof below.
\end{rem}

\begin{rem}
Theorem~\ref{HGWloop} yields an explicit formula for Hodge--Gromov--Witten invariants of $X$ as a sum over dual graphs.  Indeed, such a formula is known for every smooth toric DM stack, see \textit{e.g.}~\cite{Liu}, and~$\widetilde{\cP}$ is such an item.
\end{rem}

\begin{proof}[Proof of Theorem~\ref{HGWloop}]
The proof is similar to that of Theorem~\ref{HGWchain}.  We take the degree $d$ hypersurface $X$ to be the zero locus of the loop polynomial
$$P = x_1^{a_1}x_2+\dotsb+x_{N-1}^{a_{N-1}}x_N+x_N^{a_N}x_1,$$
and we define the family
$$\cX = \left\lbrace x_1^{a_1}x_2+\dotsb+x_{N-1}^{a_{N-1}}x_N+x_N^{a_N}x_1s=0 \right\rbrace \subset \PP(w_1,\dotsc,w_N) \times \aA^1.$$ It is endowed with a $\CC^*$-action with weight $p_j$ on $x_j$ and $p_{N+1}$ on $s$ satisfying $p_1=1$ and $p_{j+1} = (-a_1) \dotsm (-a_j)$ for $1 \leq j \leq N-1$ and $p_{N+1} = (-a_1) \dotsm (-a_j)-1$.  Moreover, the fiber $X_1$ at $s=1$ equals the smooth hypersurface $X$.  However, the DM stack $\cX$ is not smooth, as it is singular at the point $(0,\dotsc,0,1)$ of the central fiber $s=0$.  We thus need to resolve the singularities, and that is why we use the weighted blow-up.

Define the family $\widetilde{\cX}$ over $\aA^1$ as the weighted blow-up of $((0,\dotsc,0,1),s=0)$ with weights $(b_1,\dotsc,b_{N-1})$ on the variables $x_1, \dotsc, x_{N-1}$ in the chart $x_N=1$ and weight $b_s$ on $s$.  We claim that the total space of the family $\widetilde{\cX}$ is regular.

Indeed, in the local chart where $x_N=1$, it is defined by the equation
$$x_1^{a_1}x_2+\dotsb+x_{N-1}^{a_{N-1}}+x_1 s=0.$$
The choice of weights $b_1,\dotsc,b_{N-1}, b_s$ is such that the polynomial
$$x_2^{a_2}x_3+\dotsb+x_{N-1}^{a_{N-1}}+x_1 s$$
is quasi-homogeneous with these weights.
In the blow-up chart associated to the variable $x_k$, we have new variables $\dot{x}_1,\dotsc,\dot{x}_{N-1},u,$ and $\dot{s}$ satisfying $\dot{x}_k=1$ and
$$x_k = u^{b_k},\quad x_j = u^{b_j} \dot{x}_j,\quad s=u^{b_s}\dot{s},$$
and the equation defining $\widetilde{\cX}$ becomes
$$u^{a_1b_1+b_2-\delta}\dot{x}_1^{a_1}\dot{x}_2+\dot{x}_2^{a_2}\dot{x}_3+\dotsb+\dot{x}_{N-1}^{a_{N-1}}+\dot{x}_1 \dot{s}=0,$$
where the power $a_1 b_1+b_2-\delta = (a_1-1)b_1+b_2$ is positive.  In particular, we see that we can change the chart and assume $k \neq N-1$.  Therefore, the partial derivative, along $\dot{x}_{k+1}$ if $k\neq 1$ and along $\dot{s}$ if $k=1$, does not vanish.  Similarly, in the blow-up chart associated to the variable $s$, we take the partial derivative along $\dot{x}_1$.  Hence $\widetilde{\cX}$ is a smooth DM stack.

It remains to check the assumptions of Theorem~\ref{big theo} for the embedding $\widetilde{\cX} \hookrightarrow \widetilde{\cP}$.  The $\CC^*$-fixed loci for~$\cX$ and for $\widetilde{\cP}$ are the same; \textit{i.e.}~they are given by $2N-1$ isolated points at the central fiber.  Indeed, outside the exceptional divisor $E$ in $\widetilde{\cP}$, the fixed points are the $N-1$ coordinate points $(1,0,\dotsc,0), \dotsc,(0,\dotsc,0,1,0)$.  The exceptional divisor $E$ is isomorphic to the weighted projective space $\PP(b_1,\dotsc,b_{N-1},b_s)$ and carries the non-trivial $\CC^*$-action with weight $t_j-t_N$ on the $\supth{j}$ variable and $t_{N+1}$ on the last variable.  We check that
$$b_s (t_i-t_N) \neq b_i t_{N+1} \quad \textrm{and} \quad b_j (t_i-t_N) \neq b_i (t_j-t_N),\quad \forall 1 \leq i<j \leq N-1,$$
which implies that the fixed locus in the exceptional divisor consists of its $N$ coordinate points.
We see that all these $2N-1$ points belong to $\cX$.

The normal bundle of $\cX \hookrightarrow \widetilde{\cP}$ is the pull-back of the $T$-equivariant line bundle $\cO(d-E)$ on $\widetilde{\cP}$; see Remarks~\ref{weightO(d)2} and~\ref{afterloc2}.
It remains to prove the convexity up to two markings, which follows from the same positivity argument as in the proof of Theorem~\ref{HGWchain}.
Indeed, any stable map $\widetilde{f} \colon \cC \to \widetilde{\cP}$ induces a stable map $f \colon \cC \to \PP(\underline{w}) \times \aA^1$, and we have three cases:
\begin{itemize}
\item The image $f(\cC)$ does not contain the base locus $((0,\dotsc,0,1),s=0)$, so  it is isomorphic to the image $L:=\widetilde{f}(\cC)$ and we have ${\cO(d-E)}_L =\cO_L(d)$.
\item The image $f(\cC)$ is a curve containing the base locus $((0,\dotsc,0,1),s=0)$, so  the image $L:=\widetilde{f}(\cC)$ is its strict transform, in which case we have ${\cO(d-E)}_L=\cO_L(d-1)$,
\item The image $f(\cC)$ equals the base locus $((0,\dotsc,0,1),s=0)$, so  the image $L:=\widetilde{f}(\cC)$ is contained in the exceptional divisor $E$, in which case we have ${\cO(d-E)}_L={\cO(-E)}_L=\cO_L(1)$.
\end{itemize}
The degree of the line bundle on the image $\widetilde{f}(\cC)$ is non-negative in all three cases, and we have seen in the proof of Theorem~\ref{HGWchain} that this implies $H^1(\cC,\widetilde{f}^*\cO(d-E))=0$.
\end{proof}

As a special case of Theorem~\ref{HGWloop}, we obtain a full genus zero computation of the Gromov--Witten theory of loop hypersurfaces with ambient insertions.

\begin{cor}[Genus zero Gromov--Witten theory of loop hypersurfaces]\label{0HGWloop}
In the notation of Theorem~\ref{HGWloop} and under its assumptions, but with $g=0$, we have the following equality in the Chow ring of\, $\overline{\cM}_{0,n}$: 
$$\left[\cM_{0,\underline{\rho}}\(\widetilde{\cP},\beta\)\right]^{\vir,T} \cdot e_T (R \pi_*f^*\cO(d-E)) \cdot \alpha \xrightarrow[t \to 0]{} \left[\cM_{0,\underline{\rho}}(X,\beta)\right]^\vir \cdot \alpha.$$
\end{cor}

\subsection{Hodge--Gromov--Witten theory for invertible polynomials}\label{secinv}
A quasi-homogeneous polynomial $P$ is called invertible if it has as many monomials as variables.  By \cite{KS}, an invertible polynomial has an isolated singularity at the origin if and only if it is the Thom--Sebastiani sum of chain and loop polynomials.\footnote{A Fermat monomial is a special case of a chain or loop polynomial.}  Precisely, up to renaming variables, the polynomial $P$ equals
$$P(x_1,\dotsc,x_N) = P_1(x_1,\dotsc,x_{N_1}) + \dotsb + P_r(x_{N_{r-1}+1},\dotsc,x_N),$$
where the polynomials $P_1,\dotsc, P_r$ are either chain or loop polynomials.  In this section, we assume $P$ is not a Fermat polynomial, \textit{i.e.}~that we have $r < N$.

We introduce $r$ affine variables $s_1,\dotsc,s_r$, and we define a polynomial
$$\widehat{P}(x_1,\dotsc,x_N,s_1,\dotsc,s_r) = \widehat{P}_1(x_1,\dotsc,x_{N_1},s_1) + \dotsb + \widehat{P}_r(x_{N_{r-1}+1},\dotsc,x_N,s_r)$$
in the following way:
\begin{itemize}
\item If $P_i$ is a chain polynomial of the form $y_1^{a_1}y_2+\dotsb+y_N^{a_N}$, then we set
$\widehat{P}_i = y_1^{a_1}y_2+\dotsb+y_N^{a_N}s_i$.
\item Otherwise, $P_i$ is a loop polynomial of the form $y_1^{a_1}y_2+\dotsb+y_N^{a_N}y_1$, and we set
$\widehat{P}_i = y_1^{a_1}y_2+\dotsb+y_N^{a_N}y_1s_i$.
In that case, we say that $y_N$ is the last variable of $\widehat{P}_i$ and $y_1$ is its first variable.
Of course, there are $N$ choices to incorporate the variable $s_i$ in $P_i$, and we fix one once and for all.
\end{itemize}

Next, we define the $\aA^r$-family
$$\cX = \left\lbrace \widehat{P}=0 \right\rbrace \subset \PP(\underline{w}) \times \aA^r,$$
which is not regular.  We first determine the singular locus of $\cX$, and then we perform weighted blow-ups to resolve it.

Let $J \subset \left\lbrace 1,\dotsc,N \right\rbrace$ be the set of indices $j$ such that $x_j$ is the last variable of some polynomial $\widehat{P}_i$ associated to a loop polynomial $P_i$. Moreover, we define a function $\Phi \colon J \to \left\lbrace 1,\dotsc,N \right\rbrace$ which sends $j$ to the index of the first variable in the loop polynomial associated to $j$.

For any subset $J' \subset J$, we define the subspace
$$B_{J'} \subset \PP(\underline{w}) \times \aA^r$$
where all variables $x_j$ with $j \notin J'$ are zero and all variables $s_j$ with $j \in J'$ are zero.  We observe that the singular locus of $\cX$ equals
$$\mathrm{Sing}(\cX) = \bigcup_{J' \subset J} B_{J'} \subset \PP(\underline{w}) \times \aA^r.$$
Furthermore, we define the invertible polynomial
$$Q_{J'} \(\left\lbrace x_j\right\rbrace_{j \notin J'},\left\lbrace s_j\right\rbrace_{j \in J'}\) = {\widehat{P}(x_1,\dotsc,x_N,s_1,\dotsc,s_r)}_{\substack{|x_j=1 ~\forall j \in J' \\ s_j=1 ~\forall j \notin J'}}-\sum_{j \in J'} x_{\Phi(j)}^{a_{\Phi(j)}}x_{\Phi(j)+1},$$ 
which is then quasi-homogeneous of some degree $\delta \in \NN^*$, with respect to some positive weights $b_{J'}:=(b_{J'}(1),\dotsc,b_{J'}(N))$ on the variables $x_j, s_j$.

Consider the smooth DM stack $\widetilde{\cP}$ obtained with the following recipe:
\begin{itemize}
\item For all subset $J' \subset J$ of cardinality $1$, we blow up $B_{J'}$ in $\PP(\underline{w}) \times \aA^r$ with weights $b_{J'}$ on the respective variables.
\item The strict transforms of the subspaces $B_{J'}$ with $J' \subset J$ of cardinality $2$ are disjoint, and we blow them up with weights $b_{J'}$ on the respective variables.
\item $\dotsc$
\item The strict transforms of the subspaces $B_{J'}$ with $J' \subset J$ of cardinality $\#J-1$ are disjoint, and we blow them up with weights $b_{J'}$ on the respective variables.
\item We blow up the strict transform of $B_J$ with weights $b_{J'}$ on the respective variables.
\end{itemize}
We denote by $\widetilde{\cX} \hookrightarrow \widetilde{\cP}$ the strict transform of $\cX$ under the birational map $\widetilde{\cP} \to \PP(\underline{w}) \times \aA^r$, and by $E$ the sum of the exceptional divisors.

\begin{pro}
The $\aA^r$-family $\widetilde{\cX}$ is regular.  Moreover, the fiber at $(s_1,\dotsc,s_r)=(1,\dotsc,1)$ equals the degree $d$ hypersurface $X$ in $\PP(\underline{w})$ defined by the invertible polynomial $P$.
\end{pro}

\begin{proof}
To simplify the exposition, we write the proof for an invertible polynomial which is a sum of two loop polynomials; \textit{i.e.}~we take
$$\widehat{P}\(\underline{x},\underline{y}\) = x_1^{a_1}x_2+\dotsb+x_N^{a_N}x_1s_1+y_1^{a_1}y_2+\dotsb+y_N^{a_N}y_1s_2.$$
The general proof follows the same arguments.  Moreover, we consider only the charts where $x_N\neq 0$ or $y_N \neq 0$, because otherwise $\cX$ is already smooth.

Let us take the chart where $x_N=1$ and consider a point $p=(\underline{x},\underline{y},s_1,s_2)$ in this chart.  We can assume $x_1,\dotsc,x_{N-1}=0$, $y_1,\dotsc,y_{N-1}=0$, $s_1=0$, and $y_N s_2=0$; otherwise, $\cX$ is smooth at $p$.

Let us first assume $y_N=0$ so that we first blow up the locus $B_{\left\lbrace x_N\right\rbrace}$  and then the locus $B_{\left\lbrace x_N,y_N\right\rbrace}$.  After the first blow-up, the equation defining $\cX$ becomes
$$u^{\epsilon_1} \dot{x}_1^{a_1}\dot{x}_2+\dotsb+\dot{x}_1\dot{s}_1+\dot{y}_1^{a_1}\dot{y}_2+\dotsb+\dot{y}_N^{a_N}\dot{y}_1s_2 = 0$$
in the new coordinates (see the proof of Theorem~\ref{HGWloop}) and with $\epsilon_1$ some non-negative integer.  If one of the dotted variable other than $\dot{y}_N$ is non-zero, then this equation satisfies the smoothness criterion, and we are away from the second blow-up center $B_{\left\lbrace x_N,y_N\right\rbrace}$.  Thus, we can assume they are all zero but $\dot{y}_N=1$.  After the second blow-up, we then obtain
$$v^{\epsilon_2}u^{\epsilon_1} \ddot{x}_1^{a_1}\ddot{x}_2+\dotsb+\ddot{x}_1\ddot{s}_1+v^{\epsilon_3}\ddot{y}_1^{a_1}\ddot{y}_2+\dotsb+\ddot{y}_1\ddot{s}_2 = 0$$
in new coordinates defined in the same way as for the first blow-up and with some non-negative integers $\epsilon_2,\epsilon_3$. Then this equation satisfies the smoothness criterion.

Finally, let us consider the case where $y_N \neq 0$ and $s_2=0$.  Then we are away from the first blow-up center $B_{\left\lbrace x_N\right\rbrace}$, and we only need to perform the second blow-up.  The equation defining $\widetilde{\cX}$ is then of the form
$${v'}^{\epsilon_2} \(\ddot{x}'_1\)^{a_1}\ddot{x}'_2+\dotsb+\ddot{x}'_1\ddot{s}'_1+{v'}^{\epsilon_3}\(\ddot{y}'_1\)^{a_1}\ddot{y}'_2+\dotsb+y_N^{a_N}\ddot{y}'_1\ddot{s}'_2 = 0$$
in new coordinates, and it satisfies the smoothness criterion since $y_N \neq 0$.
\end{proof}

Consider the natural torus action of $T = \(\CC^*\)^{N+r}$ on the toric DM stack $\PP(\underline{w}) \times \aA^r$.  Since every subspace $B_{J'}$ is stable under the torus action, the weighted blow-up $\widetilde{\cP}$ is then again a toric DM stack with $T$-action.  Moreover, the normal bundle of $\widetilde{\cX} \hookrightarrow \widetilde{\cP}$ is the pull-back of the $T$-equivariant line bundle $\cO(d-E)$ on $\widetilde{\cP}$.

Let $T' := \(\CC^*\)^r$ and define an embedding $T' \hookrightarrow T$ such that the polynomial $\widehat{P}$ is $T'$-equivariant.  Precisely, if $t_1,\dotsc,t_N, \tau_1,\dotsc,\tau_r$ denote variables of $T$, then we impose $M(\underline{t},\underline{\tau})=1$ for every monomial $M$ of the polynomial $\widehat{P}$.  Equivalently, we let $T'$ act separately on each $\widehat{P}_i$, where we take the $\CC^*$-action defined for chain and loop polynomials in the previous sections.  We can also consider a generic embedding $\CC^* \hookrightarrow T'$.

As a consequence, the regular $\aA^r$-family $\widetilde{\cX}$ is $\CC^*$-equivariant (even $T'$-equivariant), and we check easily that the $\CC^*$-fixed loci in $\widetilde{\cX}$ and in $\widetilde{\cP}$ are equal and consist of isolated fixed points in the central fiber $(s_1,\dotsc,s_r)=(0,\dotsc,0)$.
Moreover, by the same positivity argument as in the proof of Theorem~\ref{HGWloop}, we see that the normal bundle $\cO(d-E)$ of $\widetilde{\cX} \hookrightarrow \widetilde{\cP}$ is convex up to two markings.
We then conclude with the following statement.

\begin{thm}[Genus zero Gromov--Witten theory of invertible hypersurfaces]\label{HGWinv}
We fix $n \geq 3$, $\beta \in \NN$, and isotropies $\underline{\rho} = (\rho_1, \dotsc, \rho_n)$ in $\PP(\underline{w})$.
We assume there exists a degree $d$ hypersurface in $\PP(\underline{w})$ defined by an invertible polynomial.
Let $X \subset \PP(\underline{w})$ be any smooth hypersurface of degree $d$ and $\alpha_1,\dotsc,\alpha_n$ be ambient cohomology classes on $X$, \textit{i.e.}~pulled back from $\PP(\underline{w})$.
We set $\alpha := \prod_{i=1}^n \ev_i^*(\alpha_i)$ to be the product of insertions.
Then we have the following equality in the Chow ring of\, $\overline{\cM}_{0,n}$: 
$$\left[\cM_{0,\underline{\rho}}\(\widetilde{\cP},\beta\)\right]^{\vir,T} \cdot e_T (R \pi_*f^*\cO(d-E)) \cdot \alpha \xrightarrow[t \to 0]{} \left[\cM_{0,\underline{\rho}}(X,\beta)\right]^\vir \cdot \alpha.$$
Precisely, the class $e_T (R \pi_*f^*\cO(d-E))$ is only defined after localization, so we first apply the virtual localization formula to the left-hand side, then we compute it in $A^T_*(\overline{\cM}_{0,n} \times \aA^r)$ as a formal series in the $T$-equivariant parameters and their inverses, then we specialize them via the embedding $\CC^* \hookrightarrow T$ and obtain a well-defined polynomial in $t$, and finally we take the constant coefficient and pull it back from $A_*(\overline{\cM}_{0,n} \times \aA^r)$ to $A_*(\overline{\cM}_{0,n})$.
\end{thm}

\begin{rem}
Theorem~\ref{HGWinv} is also valid in higher genus, but that case is not interesting as the right-hand side is multiplied by the $\supth{r}$ power of the Euler class of the Hodge bundle, and we know its square is zero for positive genus.
\end{rem}

\begin{rem}
Theorem~\ref{HGWinv} yields an explicit formula for genus zero Gromov--Witten invariants of $X$ as a sum over dual graphs.
Indeed, such a formula is known for every smooth toric DM stack, see \textit{e.g.}~\cite{Liu}, and~$\widetilde{\cP}$ is such an item.
Moreover, the complexity of the space $\widetilde{\cP}$ only comes from loop polynomials.
In particular, when it is possible to represent the hypersurface $X$ by a Thom--Sebastiani sum of chain polynomials, we can simplify the formula to
$$\left[\cM_{0,\underline{\rho}}(\PP(\underline{w}),\beta)\right]^{\vir,T} \cdot e_T (R \pi_*f^*\cO(d)) \cdot \alpha \xrightarrow[t \to 0]{} \left[\cM_{0,\underline{\rho}}(X,\beta)\right]^\vir \cdot \alpha$$
in the Chow ring of $\overline{\cM}_{0,n}$.
\end{rem}

\begin{rem}
There is a list of all $7555$ Calabi--Yau $3$-folds that are hypersurfaces in weighted projective spaces; see \cite{Candelas2,Klemm2,Machine} or Kreuzer's webpage.
Among them, there are about $800$ hypersurfaces represented by a Fermat polynomial, hence satisfying the convexity assumption.
There are about $6000$ hypersurfaces defined by an invertible polynomial, hence computable via Theorem~\ref{HGWinv}.
The last $10\%$ correspond to non-degenerate polynomials with more than five monomials and are not treated by this paper, \textit{e.g.}
$$x_1^{15}+x_2^5+x_3^5x_5+x_4^2x_5+x_5^9+x_3^2x_2x_4=0 \quad \textrm{in }\PP(3,9,8,20,5).$$
Observe as well that the main difficulty in Theorem~\ref{HGWinv} comes from the sequence of weighted blow-ups in the definition of the ambient space $\widetilde{\cP}$.
However, as we consider $3$-folds, there are at most two loop polynomials in the Thom--Sebastiani sum, so  the birational map $\widetilde{\cP} \to \PP(\underline{w}) \times \aA^r$ is defined by at most three blow-ups.
\end{rem}

\begin{rem}\label{chi6}
Calabi--Yau $3$-folds with Euler characteristic equal to $\chi=\pm 6$ are especially important in string theory.
The first instance is the Tian--Yau manifold, defined as the quotient of a smooth complete intersection in $\PP^3 \times \PP^3$ of degrees $(3,0), (0,3),$ and $(1,1)$ by a free action of $\ZZ/3\ZZ$.
It appeared as a candidate for a string theory's potential solution to the universe; see \cite{Machine,Greene}.
Other examples are given by hypersurfaces in weighted projective space.
A list of $40$ items is given in \cite{Klemm2}, among which $14$ hypersurfaces are defined by an invertible polynomial.
None of these $14$ hypersurface satisfy the convexity assumption, and only one involves a loop polynomial, namely
$$x_1^7x_2+x_2^5x_1+x_3^{17}x_4+x_4^2x_5+x_5^3=0 \quad \textrm{in }\PP(6,9,2,17,17).$$
\end{rem}

\subsection{Regularizable stacks}
As a by-product of Section~\ref{sec1}, we extend genus zero Gromov--Witten theory to a particular set of singular DM stacks that we call regularizable, and we prove invariance under regular deformations.

\begin{dfn}\label{regularisable}
A DM stack $X$ is called regularizable if there is an embedding $X \hookrightarrow \cX$ as a fiber in a family $\cX$ over $\aA^m$, for some integer $m$, whose total space is regular.
The genus zero Gromov--Witten theory of $X$ is then defined using a 
regularized virtual cycle.
A substack $X \subset \cP$ of a  smooth DM stack $\cP$ is called regularizable inside $\cP$ if we can choose the family $\cX$ above as a subfamily of the trivial family $\cP \times \aA^m$.
\end{dfn}

\begin{exa*}
Smooth DM stacks are regularizable via a trivial family.
The hypersurface $X_0$ from Remark~\ref{afterloc} is singular, but it is regularizable inside $\PP(\underline{w})$.
Every hypersurface in a projective space is regularizable inside the projective space.
The quartic orbifold curve
$$\left\lbrace x^4y+y^3z=0 \right\rbrace \subset \PP(1,1,2)$$
is not regularizable inside $\PP(1,1,2)$.
\end{exa*}

In the following proposition, we illustrate the fundamental role of hypersurfaces defined by a Thom--Sebastiani sum of chain polynomials.

\begin{pro}\label{caractreg}
Let $X \subset \PP(\underline{w})$ be a regularizable hypersurface inside a weighted projective space, such that there is a $\CC^*$-action on $\PP(\underline{w})$ leaving $X$ invariant and whose fixed points are isolated.
Then $X$ is singular, and there exists a smooth hypersurface in $\PP(\underline{w})$ defined by a Thom--Sebastiani sum of chain polynomials.

Conversely, if there exists a smooth hypersurface in $\PP(\underline{w})$ defined by chain polynomials, then there is a regularizable hypersurface stable under a $\CC^*$-action from $\PP(\underline{w})$ with isolated fixed points.
\end{pro}

\begin{proof}
For every hypersurface $X = \left\lbrace P=0 \right\rbrace \subset \PP(\underline{w})$, denoting by $\cM$ the set of monomials of $P$, we see easily that
\begin{align*}
(1,0\dotsc,0) \notin X & \iff  \exists m \in \NN^*,\quad x_1^m \in \cM, \\
(1,0\dotsc,0) \in X-\mathrm{Sing}(X) & \iff  \exists m \in \NN^*, j \neq 1,\quad x_1^mx_j \in \cM.
\end{align*}
Moreover, if $(1,0\dotsc,0) \in \mathrm{Sing}(X)$, then we have
$$\textrm{$X$ is regularizable inside $\PP(\underline{w})$} \implies w_1 | d.$$
Therefore, whenever $X$ is regularizable inside $\PP(\underline{w})$, we have, for every $1 \leq j \leq N$, either $w_j | d$ or a monomial $x_j^{a_j}x_k \in \cM$, with possibly $k=j$.

Furthermore, assume $X$ is invariant under a $\CC^*$-action on $\PP(\underline{w})$ with weights $p_1,\dotsc,p_N$.
If we have $(1,0,\dotsc,0) \notin X$, then we get $w_1 p_j=w_j p_1$ for all variable $x_j$ involved in the polynomial $P$, and fixed points are not isolated (unless $P$ is a Fermat monomial).
Thus, if the $\CC^*$-action has only isolated fixed points, there are no Fermat monomials in $\cM$ (unless $P$ is itself a Fermat monomial, and then there is a smooth Fermat hypersurface in $\PP(\underline{w})$).

As a consequence, if $X$ is as in the statement, then it contains all coordinate points.
We introduce the set $T \subset \left\lbrace 1,\dotsc,N\right\rbrace^2$ defined by
$$(i,j) \in T \iff \exists a_i \in \NN^*,\quad x_i^{a_i}x_j \in \cM.$$
By the discussion above, $T$ does not intersect the diagonal, and for every $i$ such that $w_i$ does not divide~$d$, there is at least one $j$ such that $(i,j) \in T$.
We view $T$ as a directed graph, and we check easily that if we have a loop in $T$, \textit{i.e.}~$j_1, \dotsc, j_m$ such that $(j_1,j_2), \dotsc, (j_m,j_1)$ are in $T$, then the $\CC^*$-action is trivial on $\PP(w_{j_1},\dotsc,w_{j_m}) \subset \PP(\underline{w})$.
Moreover, if we have two edges colliding, \textit{i.e.}~$i,j,k$ such that $(i,j)$ and $(k,j)$ are in $T$, then $w_j p_k = w_k p_j$ and the $\CC^*$-action is trivial on $\PP(w_j,w_k) \subset \PP(\underline{w})$.
Therefore, there is a directed subgraph in $T$ consisting of a disjoint union of directed lines. Each line corresponds to a chain polynomial without its last Fermat monomial; the statement follows.

Conversely, we take a Thom--Sebastiani sum of chain polynomials $\widehat{P}$.
Up to renaming variables, we can assume the Fermat monomials of $\widetilde{P}$ are $y_1^{b_1}, \dotsc, y_m^{b_m}$.
If $\widehat{P}$ is not a Fermat polynomial, \textit{i.e.}~$m \neq N$, then we define
$$\widetilde{P} = \widehat{P}+(s_1-1)y_1^{b_1} + \dotsb + (s_m-1)y_m^{b_m} \quad \textrm{and} \quad P=\widetilde{P}_{|s_1=\cdots=s_m=0}.$$
Then the singular hypersurface $X = \left\lbrace P=0 \right\rbrace \subset \PP(\underline{w})$ is invariant under a $\CC^*$-action on $\PP(\underline{w})$ whose fixed points are isolated, and the family $\cX = \left\lbrace \widetilde{P} = 0 \right\rbrace \subset \PP(\underline{w}) \times \aA^m$ is regular.
If $\widehat{P}$ is the Fermat polynomial $\widehat{P} = y_1^{b_1}+ \dotsb+ y_N^{b_N}$, then we define
$$\widetilde{P} = y_1^{b_1}s_1+ \dotsb+ y_{N-1}^{b_{N-1}}s_{N-1} + y_N^{b_N}$$
and $P=y_N^{b_N}$. Then the singular hypersurface $X = \left\lbrace P=0 \right\rbrace \subset \PP(\underline{w})$ is invariant under a $\CC^*$-action on $\PP(\underline{w})$ whose fixed points are isolated, and the family $\cX = \left\lbrace \widetilde{P} = 0 \right\rbrace \subset \PP(\underline{w}) \times \aA^{N-1}$ is regular.
\end{proof}

\begin{pro}\label{invariance}
Let $X$ be a regularizable DM stack, and let $\cX$ be a family over $\aA^m$ whose total space is regular and such that it contains $X$ as a fiber.
Then every fiber is a regularizable DM stack, and the genus zero Gromov--Witten theory is independent of the fiber.
\qed
\end{pro}

\begin{rem}
A special feature of genus zero Gromov--Witten theory is that we do not need a globally defined torus action on the family over $\aA^m$ to apply the equivariant regular specialization theorem (Theorem~\ref{skid equiv}): a torus action on the central fiber is enough, as soon as the normal bundle of the central fiber is torus equivariant.
\end{rem}


\newcommand{\etalchar}[1]{$^{#1}$}

\end{document}